\documentclass[12pt, reqno]{amsart}
\usepackage{amsmath, amstext, amsbsy, amssymb, ytableau}

\newtheorem{theorem}{Theorem}[section]
\newtheorem{lemma}[theorem]{Lemma}
\newtheorem{proposition}[theorem]{Proposition}

\theoremstyle{definition}
\newtheorem{definition}[theorem]{Definition}

\theoremstyle{remark}
\newtheorem{remark}[theorem]{Remark}
\newtheorem{conjecture}[theorem]{Conjecture}

\numberwithin{equation}{section}

\newcommand{\BD}{\mathcal {BD}}

\newcommand{\C}{ \mathbb C }
\newcommand{\Coe}{ {\rm Coeff} }

\newcommand{\MD}{\mathcal {MD}}
\newcommand{\N}{ \mathbb N }

\newcommand{\Q}{ \mathbb Q }
\newcommand{\qBD}{{\rm q}{\mathcal {BD}}}
\newcommand{\qMD}{{\rm q}{\mathcal {MD}}}
\newcommand{\qMZV}{ {\bf qMZV} }

\newcommand{\w}{\tilde}
\newcommand{\W}{\widetilde}

\newcommand{\Z}{ \mathbb Z }

\def\beq{\begin{equation}}
\def\eeq{\end{equation}}

\numberwithin{equation}{section}

\begin{document}

\title[Multiple $q$-zeta values and traces]
      {Multiple $q$-zeta values and traces}

\author[Zhenbo Qin]{Zhenbo Qin}
\address{Department of Mathematics, University of Missouri, 
         Columbia, MO 65211, USA} 
\email{qinz@missouri.edu}

\date{\today}

\begin{abstract}
Let $(a)_\infty = (a; q)_\infty = \prod_{n=0}^\infty (1-aq^n)$. 
An elegant result of Bloch and Okounkov \cite{BO} states that 
if $x = e^z$, then 
$$
\frac{(xq)_\infty (x^{-1}q)_\infty}{(q)_\infty^2}, 
$$
which appears in various traces in representation theory and 
algebraic geometry, is a formal power series in $z^2$ whose 
coefficient for $z^{2k}$ is a quasi-modular form of weight $2k$. 
Quasi-modular forms are special types of multiple $q$-zeta values. 
In this paper, we generalize this result of Bloch and Okounkov and 
prove that certain other traces are related to multiple $q$-zeta values. 
A simple case of our main results asserts that 
if $x = e^z$ and $y = e^w$, then 
$$
\frac{(xq)_\infty (yq)_\infty}{(q)_\infty (xyq)_\infty}, 
$$
which appears in \cite[Theorem~5]{CW} as a trace (the deformed 
Bloch-Okounkov $1$-point function), 
is a formal power series in $z$ and $w$ whose coefficient for  
$z^mw^n$ is a multiple $q$-zeta value (in the sense of 
\cite{BK3, Oko}) of weight $(m+n)$.
\end{abstract}

\maketitle

\section{\bf Introduction} 
\label{sect_intr}

In the region ${\rm Re} \, s > 1$, the Riemann zeta function is defined by 
$$
\zeta(s) = \sum_{n =1}^\infty \frac{1}{n^{s}}.
$$
The integers $s > 1$ give rise to a sequence of special values of the Riemann zeta function.
Multiple zeta values are series of the form
$$
\zeta(s_1, \ldots, s_k) = \sum_{n_1 > \cdots > n_k > 0} 
\frac{1}{n_1^{s_1} \cdots n_k^{s_k}}
$$
where $s_1, \ldots, s_k$ are positive integers with $s_1 > 1$, 
and $n_1, \ldots, n_k$ denote positive integers.
Multiple $q$-zeta values are $q$-deformations of 
$\zeta(s_1, \ldots, s_k)$, which may take different forms 
\cite{Bac, BK1, BK2, BK3, Bra1, Bra2, Oko, OT, Zhao, Zud}. 
Following \cite{Bac, BK2}, the $\Q$-vector space spanned by 
the weight-$\sum_{i=1}^\ell (s_i +r_i)$ bi-brackets
$$
  \begin{bmatrix} s_1, \ldots, s_\ell\\r_1, \ldots, r_\ell \end{bmatrix}
= \sum_{\substack{u_1 > \cdots > u_\ell > 0\\v_1, \ldots, v_\ell > 0}}
\frac{u_1^{r_1} \cdots u_\ell^{r_\ell}}{r_1! \cdots r_\ell!} \cdot
\frac{v_1^{s_1-1} \cdots v_\ell^{s_\ell-1}}{(s_1-1)! \cdots (s_\ell-1)!} 
\cdot q^{u_1v_1 + \cdots + u_\ell v_\ell} \, \in \, \Q[[q]],
$$
where $\ell \ge 0, s_1, \ldots, s_\ell \ge 1$ and 
$r_1, \ldots, r_\ell \ge 0$ are integers, is denoted by $\BD$. 
Define $\qBD$ (respectively, $\MD$) to be the subspace of $\BD$ spanned 
by all the bi-brackets with $s_1 > 1$ (respectively, 
with $r_1 = \ldots = r_\ell = 0$). Let 
$$
\qMD = \MD \cap \qBD.
$$
The multiple $q$-zeta values defined by Okounkov \cite{Oko} 
are denoted by $Z(s_1, \ldots, s_k)$ where $s_1, \ldots, s_k > 1$ 
are integers (see Definition~\ref{def_qMZV}~(iv)). 
The $\Q$-linear span of all the multiple $q$-zeta values 
$Z(s_1, \ldots, s_k)$ with $s_1, \ldots, s_k > 1$ is denoted by $\qMZV$. 
Let ${\bf QM}$ be the set of all quasi-modular forms 
(of level $1$ on the full modular group ${\rm PSL}(2; \Z)$) over $\Q$. 
By \eqref{Comm_alg} which summarizes various results in 
\cite{Bac, BK1, BK2, BK3}, we have inclusions of $\Q$-algebras: 
\begin{eqnarray}   \label{Intro-Comm_alg}
\begin{matrix}
&&&&\qBD &\subset&\BD   \\
&&&&\cup &&\cup   \\
{\bf QM}&\subset &\qMZV&\subset &\qMD&\subset&\MD.  
\end{matrix}
\end{eqnarray}

Let $q$ be a formal variable. 
Put $(a; q)_n = (1-a)(1-aq) \cdots (1-aq^{n-1})$ and
$$
(a)_\infty = (a; q)_\infty = \prod_{n=0}^\infty (1-aq^n).
$$
An elegant result of Bloch and Okounkov \cite{BO} states that 
if $x = e^z$, then 
\begin{eqnarray}    \label{BOx}
\frac{(xq)_\infty (x^{-1}q)_\infty}{(q)_\infty^2}, 
\end{eqnarray}
which appears in various traces in representation theory and 
algebraic geometry \cite{BO, Car1, Car2, LQW}, 
is a formal power series in $z^2$ such that 
the coefficient of $z^{2k}$ is a quasi-modular form of weight $2k$. 
Our first theorem generalizes this result to multiple $q$-zeta values. 

\begin{theorem}   \label{Intro-AAA06261252pm}
Let $(a)_\infty = (a; q)_\infty = (1-a)(1-aq)(1-aq^2) \cdots$.
Fix $r \ge 1$. For $1 \le j \le r$, let $x_j = e^{z_j}$ 
and $y_j = e^{w_j}$. If $x_1 \cdots x_r = y_1 \cdots y_r$, then 
\begin{eqnarray}    \label{Intro-AAA06261252pm.0}
\prod_{j = 1}^r \frac{(x_jq)_\infty}{(y_jq)_\infty} 
\in \qMZV[[z_1, \ldots, z_r, w_1, \ldots, w_{r-1}]],
\end{eqnarray}
the constant term is $1$, and the coefficient of 
$z_1^{m_1} \cdots z_r^{m_r} w_1^{n_1} \cdots w_{r-1}^{n_{r-1}}$ has weight
$$
m_1 + \ldots + m_r + n_1 + \ldots + n_{r-1}.
$$
\end{theorem}

The product in \eqref{Intro-AAA06261252pm.0} appears throughout \cite{CW} 
which deals with deformed Bloch-Okounkov $n$-point functions and 
certain vertex operators related to Macdonald polynomials. 
A simple case of Theorem~\ref{Intro-AAA06261252pm} states that 
if $x = e^z$ and $y = e^w$, then 
$$
\frac{(xq)_\infty (yq)_\infty}{(q)_\infty (xyq)_\infty}
$$
is a formal power series in $z$ and $w$ such that the coefficient of 
$z^mw^n$ is a multiple $q$-zeta value of weight $(m+n)$. 
This is indeed a generalization of the result of Bloch and Okounkov 
mentioned above for \eqref{BOx} in view of ${\bf QM} \subset \qMZV$ 
from \eqref{Intro-Comm_alg}.

Next, we introduce the expression $\mathfrak P_{N}^{a,b}$. 
Note that $\mathfrak P_{N}^{0,0}$ with 
$x_\ell^{i,j} = 1$ for all $1 \le i < j \le N$ and $\ell \in \{1, 2\}$ 
is a special case of the product \eqref{Intro-AAA06261252pm.0}.

\begin{definition}   \label{def-zmvn}
Let $N \ge 2$ be an integer. Fix variables $y_1, \ldots, y_N, a, b$. 
For $1 \le i < j \le N$, $\ell \in \{1, 2\}$ and $1 \le r \le N$, let 
$y_{i,j} = y_i^{-1}y_j, x_\ell^{i,j} = \exp(z_\ell^{i,j}), 
u_\ell^{i,j} = \exp(v_\ell^{i,j}), x_i = \exp(z_i)$ and $u_i = \exp(v_i)$.
We define $\mathfrak P_{N}^{a,b}$ by 
\begin{eqnarray*}     
   \mathfrak P_{N}^{a,b}
&=&\prod_{1 \le i < j \le N} 
   \frac{(x_1^{i,j}x_2^{i,j}y_{i,j})_\infty (y_{i,j})_\infty}
     {(x_1^{i,j}y_{i,j})_\infty (x_2^{i,j}y_{i,j})_\infty} 
   \frac{(u_1^{i,j}u_2^{i,j}y_{i,j}^{-1}q)_\infty (y_{i,j}^{-1}q)_\infty}
     {(u_1^{i,j}y_{i,j}^{-1}q)_\infty (u_2^{i,j}y_{i,j}^{-1}q)_\infty} \\
& &\cdot \prod_{1 \le i \le N} 
   \bigg (\frac{(x_iy_i)_\infty}{(y_i)_\infty} \bigg )^a
   \bigg (\frac{(u_iy_i^{-1}q)_\infty}{(y_i^{-1}q)_\infty} \bigg )^b.
\end{eqnarray*}
regarded as an element in 
$$
\Q[[q, y_{i,j}, y_{i,j}^{-1}, z_\ell^{i,j}, v_\ell^{i,j}, z_r, v_r,
y_r, y_r^{-1}, a, b|1 \le i < j \le N, 1 \le \ell \le 2, 1 \le r \le N]].
$$
\end{definition}

The expressions $\mathfrak P_{N}^{a,b}$ appear as traces in \cite{AQ2} 
which investigates Okounkov's Conjecture \cite{Oko} connecting 
multiple $q$-zeta values and the Hilbert schemes of points on 
a smooth projective surfaces 
\cite{Alh, AQ1, Car1, Car2, CO, Qin1, Qin2, QY, Zhou}. 
Hence the results in this paper will find interesting applications 
in \cite{AQ2}. Additionally, expressions similar to 
$\mathfrak P_{N}^{a,b}$ appear in \cite[Theorem~16]{CW}. 

\begin{theorem}  \label{Intro-yij1toNm>0} 
Let notations be from Definition~\ref{def-zmvn}. Then,
\begin{enumerate}
\item[{\rm (i)}]
for $m_\ell^{i,j}, n_\ell^{i,j}, m_i, n_i \ge 0$, the coefficient of 
$$
\prod_{1 \le i < j \le N, 1 \le \ell \le 2} (z_\ell^{i,j})^{m_\ell^{i,j}}
(v_\ell^{i,j})^{n_\ell^{i,j}} \cdot \prod_{1 \le i \le N} 
z_i^{m_i} v_i^{n_i}
$$ 
in 
$
\Coe_{y_1^0 \cdots y_N^0}(\mathfrak P_{N}^{a,b})
$ 
is contained in $\BD[a, b]$ with weight at most 
\begin{eqnarray}    \label{Intro-yij1toNm>0.01}
\sum_{1 \le i < j \le N, 1 \le \ell \le 2} (m_\ell^{i,j} + n_\ell^{i,j})
+ \sum_{1 \le i \le N} (m_i + n_i)
\end{eqnarray}
and with degree in ($a$ and $b$) at most 
$
    \sum_{1 \le i \le N} (m_i + n_i)
$.

\item[{\rm (ii)}] 
for $a=b=0$ and $m_\ell^{i,j}, n_\ell^{i,j}\ge 0$, the coefficient of 
$$
\prod_{1 \le i < j \le N, 1 \le \ell \le 2} (z_\ell^{i,j})^{m_\ell^{i,j}}
(v_\ell^{i,j})^{n_\ell^{i,j}}
$$ 
in 
$
\Coe_{y_1^0 \cdots y_N^0}(\mathfrak P_{N}^{0,0})
$ 
is contained in $\qBD$ with weight at most 
\begin{eqnarray}    \label{Intro-yij1toNm>0.03}
\sum_{1 \le i < j \le N, 1 \le \ell \le 2} (m_\ell^{i,j} + n_\ell^{i,j}).
\end{eqnarray}
\end{enumerate}
\end{theorem}

The conclusions of Theorem~\ref{Intro-yij1toNm>0} is weaker than 
those of Theorem~\ref{Intro-AAA06261252pm}. It would be interesting to 
see whether Theorem~\ref{Intro-yij1toNm>0} could be sharpened 
in the sense that the algebras $\BD$ and $\qBD$ in 
Theorem~\ref{Intro-yij1toNm>0} are replaced by $\qMZV$. 
Closely related to these discussions is the conjecture asserting 
that 
$$
\BD = \mathcal {MD}
$$
(\cite[Conjecture~4.3]{Bac} and \cite[Conjecture~3.2~(B2)]{BK2}). 
In the same spirit, we propose the following conjecture. 

\begin{conjecture}   \label{Conj-Qin} 
\begin{enumerate}
\item[{\rm (i)}] $\qBD = \qMD$. 

\item[{\rm (ii)}] $\qMZV$ coincides with the subspace 
of $\BD$ spanned by all the bi-brackets
$$
\begin{bmatrix} s_1, \ldots, s_\ell\\r_1, \ldots, r_\ell \end{bmatrix},
\qquad \ell \ge 0, s_1, \ldots, s_\ell \ge 2.
$$
\end{enumerate}
\end{conjecture}

Section~\ref{sect_Basic_def} recalls basic definitions and notations 
about multiple $q$-zeta values. Theorem~\ref{Intro-AAA06261252pm} 
(= Theorem~\ref{AAA06261252pm}) is proved in Section~\ref{sect_Expressing}. 
Section~\ref{sect_1to2} verifies Theorem~\ref{Intro-yij1toNm>0} 
for the basic case $\mathfrak P_{2}^{0,0}$. We confirm 
Theorem~\ref{Intro-yij1toNm>0} (= Theorem~\ref{yij1toNm>0}) in 
Section~\ref{sect_mpositive}. The main ideas in our proofs are to 
use inductions, sums of powers, and formula \cite[(6.9)]{BO} 
which expresses the higher derivatives of $e^{f(z)}$ in terms of 
those of $f(z)$. 


\section{\bf Basic definitions and notations of multiple $q$-zeta values} 
\label{sect_Basic_def}

In this section, we will review basic definitions and notations of 
multiple $q$-zeta values. Moreover, we will recall and propose 
some conjectures relating brackets and bi-brackets. 

The following definitions are from \cite{BK1, BK3, Oko}.

\begin{definition}  \label{def_qMZV}
Let $\N = \{1,2,3, \ldots\}$, and fix a subset $S \subset \N$. 
\begin{enumerate}
\item[{\rm (i)}]
Let $Q = \{Q_s(t)\}_{s \in S}$ where each $Q_s(t) \in \Q[t]$ is a polynomial 
with $Q_s(0) = 0$ and $Q_s(1) \ne 0$. For $s_1, \ldots, s_\ell \in S$ 
with $\ell \ge 1$, define
$$
Z_Q(s_1, \ldots, s_\ell)
= \sum_{n_1 > \cdots > n_\ell \ge 1} \prod_{i=1}^\ell 
  \frac{Q_{s_i}(q^{n_i})}{(1-q^{n_i})^{s_i}}
\in \Q[[q]].
$$
Put $Z_Q(\emptyset) = 1$, and define $Z(Q, S)$ to be the $\Q$-linear span 
of the set
$$
\{Z_Q(s_1, \ldots, s_\ell)| \, \ell \ge 0 \text{ and }
s_1, \ldots, s_\ell \in S \}.
$$

\item[{\rm (ii)}]
Define $\mathcal {MD} = Z(Q^E, \N)$ where $Q^E = \{Q_s^E(t)\}_{s \in \N}$
and 
$$
Q_s^E(t) = \frac{tP_{s-1}(t)}{(s-1)!}
$$ 
with $P_s(t)$ being the Eulerian polynomial defined by 
\begin{eqnarray}   \label{def_qMZV.01}
\frac{tP_{s-1}(t)}{(1-t)^s} = \sum_{d=1}^\infty d^{s-1}t^d.
\end{eqnarray}  
We have $t P_{0}(t) = t$.
For $s > 1$, the polynomial $t P_{s-1}(t)$ has degree $s - 1$. 
For $s_1, \ldots, s_\ell \in \N$ with $\ell \ge 1$, define
\begin{eqnarray}   \label{def_qMZV.02}
[s_1, \ldots, s_\ell] 
= Z_{Q^E}(s_1, \ldots, s_\ell)
= \sum_{n_1 > \cdots > n_\ell \ge 1} \prod_{i=1}^\ell 
  \frac{Q_{s_i}^E(q^{n_i})}{(1-q^{n_i})^{s_i}}.
\end{eqnarray}  

\item[{\rm (iii)}]
Define $\text{q}\mathcal {MD}$ be the subspace of $\mathcal {MD}$ linearly 
spanned by $1$ and all the brackets $[s_1, \ldots, s_\ell]$ with $s_1 > 1$. 

\item[{\rm (iv)}]
Define $\qMZV = Z(Q^O, \N_{>1})$ where 
$Q^O = \{Q_s^O(t)\}_{s \in \N_{>1}}$ with 
$$
  Q_s^O(t) 
= \begin{cases}
  t^{s/2},           &\text{if $s \ge 2$ is even,} \\
  t^{(s-1)/2} (t+1), &\text{if $s \ge 3$ is odd,}
  \end{cases}
$$
and $\N_{>1} = \{2,3,4, \ldots\}$.
For $s_1, \ldots, s_\ell \in \N$ with $\ell \ge 1$, define
$$
Z(s_1, \ldots, s_\ell) = Z_{Q^O}(s_1, \ldots, s_\ell).
$$
\end{enumerate}
\end{definition}

We see from Definition~\ref{def_qMZV}~(ii) that for $s \ge 1$, 
\begin{eqnarray}   \label{def_qMZV.03}
(s-1)! \cdot \frac{Q_{s}^E(t)}{(1-t)^s} = \sum_{d=1}^\infty d^{s-1}t^d.
\end{eqnarray} 
It follows that 
\begin{eqnarray*}   
  (s-1)! \cdot [s] 
= (s-1)! \cdot \sum_{n \ge 1} \frac{Q_{s}^E(q^n)}{(1-q^n)^s} 
= \sum_{n \ge 1} \sum_{d=1}^\infty d^{s-1} q^{nd}
\end{eqnarray*} 
for $s \ge 1$. Therefore, we obtain
\begin{eqnarray}   \label{def_qMZV.04}
  [s] 
= \frac{1}{(s-1)!} \sum_{n, d \ge 1} d^{s-1} q^{nd}
= \frac{1}{(s-1)!} \sum_{d \ge 1} d^{s-1} \frac{q^{d}}{1-q^d}.
\end{eqnarray}  
More generally, for $s_1, \ldots, s_\ell \ge 1$, we have
\begin{eqnarray}   \label{def_qMZV.05}
  [s_1, \ldots, s_\ell] 
= \frac{1}{(s_1-1)! \cdots (s_\ell-1)!} 
  \sum_{\substack{n_1 > \cdots > n_\ell \ge 1\\d_1, \ldots, d_\ell \ge 1}} 
  d_1^{s_1-1} \cdots d_\ell^{s_\ell-1} q^{n_1d_1+\ldots+n_\ell d_\ell}.
\end{eqnarray} 

By the Theorem~2.13 and Theorem~2.14 in \cite{BK1},
$\text{q}\mathcal {MD}$ is a subalgebra of $\mathcal {MD}$, 
and $\mathcal {MD}$ is a polynomial ring over $\text{q}\mathcal {MD}$ 
with indeterminate $[1]$:
\begin{eqnarray}   \label{BK1Thm214}
\mathcal {MD} = \text{q}\mathcal {MD}[\,[1]\,].
\end{eqnarray}
By the Proposition~2.2 and Theorem~2.4 in \cite{BK3}, 
$Z(\{Q_s^E(t)\}_{s \in \N_{>1}}, \N_{>1})$ is a subalgebra of 
$\mathcal {MD}$ as well and 
$\qMZV = Z(\{Q_s^E(t)\}_{s \in \N_{>1}}, \N_{>1})$. Therefore,   
\begin{eqnarray}   \label{BK3-2.4}
\qMZV = Z(\{Q_s^E(t)\}_{s \in \N_{>1}}, \N_{>1}) \subset \qMD \subset \MD
\end{eqnarray} 
are inclusions of $\Q$-algebras.  
For instance, by \cite[Example~2.6]{BK3}, 
\begin{eqnarray}   \label{BK3-2.6}
Z(2) = [2], \quad Z(3) = 2[3], \quad Z(4) = [4] - \frac16 [2].
\end{eqnarray}

The graded ring ${\bf QM}$ of quasi-modular forms (of level $1$  
on the full modular group ${\rm PSL}(2; \Z)$) over $\Q$ is 
the polynomial ring over $\Q$ generated by the Eisenstein series 
$G_2(q), G_4(q)$ and $G_6(q)$:
$$
{\bf QM} = \Q[G_2, G_4, G_6] = {\bf M}[G_2]
$$
where ${\bf M} = \Q[G_4, G_6]$ is the graded ring of modular forms 
(of level $1$) over $\Q$, and 
\begin{eqnarray*}    
G_{2k} = G_{2k}(q) =\frac{1}{(2k-1)!} \cdot 
\left (-\frac{B_{2k}}{4k} + \sum_{n \ge 1} 
\Big ( \sum_{d|n} d^{2k-1} \Big )q^n \right )
\end{eqnarray*}
where $B_i \in \Q, i \ge 2$ are the Bernoulli numbers defined by 
$$
\frac{t}{e^t - 1} = 1 - \frac{t}{2} + \sum_{i =2}^{+\infty} B_i \cdot \frac{t^i}{i!}.
$$
The grading is to assign $G_2, G_4, G_6$ weights $2, 4, 6$ respectively.
By \cite[p.6]{BK3}, 
\begin{eqnarray*}
G_2 &=& -\frac{1}{24} + Z(2),      \\
G_4 &=& \frac{1}{1440} + Z(2) + \frac{1}{6}Z(4),     \\
G_6 &=& -\frac{1}{60480} + \frac{1}{120} Z(2) + \frac{1}{4} Z(4) + Z(6).
\end{eqnarray*}
It follows that
\begin{eqnarray}  \label{20170812522pm} 
{\bf QM} = \Q[Z(2), Z(4), Z(6)] \subset \qMZV.
\end{eqnarray}

Next, we recall bi-brackets from \cite{Bac, BK2}.

\begin{definition}   \label{def-BiBracket}
\begin{enumerate}
\item[{\rm (i)}] 
For integers $\ell \ge 0, s_1, \ldots, s_\ell \ge 1$ and 
$r_1, \ldots, r_\ell \ge 0$, define
$$
  \begin{bmatrix} s_1, \ldots, s_\ell\\r_1, \ldots, r_\ell \end{bmatrix}
= \sum_{\substack{u_1 > \cdots > u_\ell > 0\\v_1, \ldots, v_\ell > 0}}
\frac{u_1^{r_1} \cdots u_\ell^{r_\ell}}{r_1! \cdots r_\ell!} \cdot
\frac{v_1^{s_1-1} \cdots v_\ell^{s_\ell-1}}{(s_1-1)! \cdots (s_\ell-1)!} 
\cdot q^{u_1v_1 + \cdots + u_\ell v_\ell} \, \in \, \Q[[q]].
$$
This $q$-series is called a {\it bi-bracket} of depth $\ell$ and of weight 
$\sum_{i=1}^\ell (s_i +r_i)$.

\item[{\rm (ii)}] 
Define $\BD$ to be the vector space spanned by the bi-brackets 
over $\Q$.

\item[{\rm (iii)}] 
Define $\qBD$ to be the subspace spanned by all the
bi-brackets with $s_1 > 1$. 
%
\end{enumerate}
\end{definition}

Note that 
$$
  \begin{bmatrix} s_1, \ldots, s_\ell\\0, \ldots, 0 \end{bmatrix}
= [s_1, \ldots, s_\ell].
$$
By \cite[(2.7)]{BK2}, 
\beq  \label{BK2-2.7}
  \begin{bmatrix} s_1, \ldots, s_\ell\\r_1, \ldots, r_\ell \end{bmatrix}
= \sum_{n_1 > \cdots > n_\ell > 0} \prod_{i=1}^\ell 
  \left (\frac{n_i^{r_i}}{r_i!} \cdot 
  \frac{q^{n_i}P_{s_i-1}(q^{n_i})}{(s_i-1)! \cdot (1-q^{n_i})^{s_i}} 
  \right ).
\eeq
By \cite[Theorem~A~i)]{Bac}, $\BD$ is a $\Q$-algebra and contains 
$\MD$ as a subalgebra. Similar proofs show that 
$\qBD$ is a subalgebra of $\BD$ and contains $\qMD$ as a subalgebra. 
Together with \eqref{BK3-2.4} and \eqref{20170812522pm}, 
we obtain inclusions of $\Q$-algebras:
\begin{eqnarray}   \label{Comm_alg}  
\begin{matrix}
&&&&\qBD &\subset&\BD   \\
&&&&\cup &&\cup   \\
{\bf QM}&\subset &\qMZV&\subset &\qMD&\subset&\MD.  
\end{matrix}
\end{eqnarray} 

\section{\bf Proof of Theorem~\ref{Intro-AAA06261252pm}} 
\label{sect_Expressing}

In this section, we prove Theorem~\ref{Intro-AAA06261252pm} 
(= Theorem~\ref{AAA06261252pm}). To begin with, we verify  
a special case of Theorem~\ref{Intro-AAA06261252pm} by showing that  
$$
\frac{(q)_\infty (xyq)_\infty}{(xq)_\infty (yq)_\infty} \in 1 
+ \qMZV[[z, w]] \cdot zw
$$
where $(a)_\infty = (a; q)_\infty$, $x = e^{z}$ and $y = e^{w}$. 
This special case together with an induction yields the proof of 
Theorem~\ref{Intro-AAA06261252pm}.

\begin{lemma}   \label{AAA06250955am}
Let $(a)_\infty = (a; q)_\infty$, $x = e^{z}$ and $y = e^{w}$. Then,
$$
\frac{(q)_\infty (xyq)_\infty}{(xq)_\infty (yq)_\infty} \in 1 
+ \qMZV[[z, w]] \cdot zw.
$$
Moreover, for $m, n \ge 1$, the coefficient of $z^m w^n$ has weight $m+n$.
\end{lemma}
\begin{proof}
For $m \ge 0$, let $a_{m} = \Coe_{z^m} 
\big ((q)_\infty (xyq)_\infty (xq)_\infty^{-1} (yq)_\infty^{-1} 
\big ) \in \C[[q, w]]$. Then,
$$
  a_{m}
= \frac{1}{m!} \frac{{\rm d}^m}{{\rm d}z^m}
  \big ((q)_\infty (xyq)_\infty (xq)_\infty^{-1} (yq)_\infty^{-1} \big )
  \Big |_{z=0}.
$$
By \cite[(6.9)]{BO}, for a function $f(z)$ of a variable $z$, we have
\begin{eqnarray}  \label{BO6.9}
  \frac{{\rm d}^m}{{\rm d}z^m} e^{f(z)} 
= e^{f(z)} \cdot m! \sum_{k_1+2k_2+3k_3+\cdots = m}
  \frac{1}{k_1! k_2! k_3! \cdots} \prod_{t \ge 1}
  \left (\frac{f^{(t)}}{t!} \right )^{k_t}. 
\end{eqnarray} 
Put 
\begin{eqnarray}  \label{Dx}
D_x = x \frac{\rm d}{{\rm d}x}.
\end{eqnarray} 
Setting $z = 0$ on both sides of \eqref{BO6.9} and 
using ${\rm d}/{\rm d}z = D_x$, we get
\begin{eqnarray}  \label{BO6.9z=0}
  \frac{{\rm d}^m}{{\rm d}z^m} e^{f(z)} \Big |_{z=0}
= e^{f(0)} \cdot m! \sum_{k_1+2k_2+3k_3+\cdots = m}
  \frac{1}{k_1! k_2! k_3! \cdots} \prod_{t \ge 1}
  \left (\frac{D_x^t f}{t!} \right )^{k_t} \Big |_{z=0}. 
\end{eqnarray}
Applying this to $f(z) = \ln\big ((q)_\infty (xyq)_\infty 
(xq)_\infty^{-1} (yq)_\infty^{-1} \big )$ yields
\begin{eqnarray}  \label{AAA06250955am.1}
  a_{m}
= \sum_{k_1+2k_2+3k_3+\cdots = m}
  \frac{1}{k_1! k_2! k_3! \cdots} \prod_{t \ge 1}
  \left (\frac{D_x^t f}{t!} \Big |_{x=1}\right )^{k_t}. 
\end{eqnarray} 
Since $f(z) = 
\ln\big ((q)_\infty (yq)_\infty^{-1} \big ) + \sum_{n \ge 1} \ln (1-xyq^n)
- \sum_{n \ge 1} \ln (1-xq^n)$,
\begin{eqnarray*} 
   D_xf  
&=&-\sum_{n \ge 1} \frac{xyq^n}{1-xyq^n} + \sum_{n \ge 1} \frac{xq^n}{1-xq^n} \\
&=&-\sum_{n \ge 1} \sum_{d \ge 1} (xyq^n)^d 
   + \sum_{n \ge 1} \sum_{d \ge 1} (xq^n)^d.
\end{eqnarray*} 
It follows that for $t \ge 1$, we have
\begin{eqnarray*} 
   D_x^tf|_{x=1}
&=&-\sum_{n \ge 1} \sum_{d \ge 1} d^{t-1} (yq^n)^d 
   + \sum_{n \ge 1} \sum_{d \ge 1} d^{t-1} q^{nd} \\
&=&-\sum_{n \ge 1} \sum_{d \ge 1} d^{t-1}q^{nd} \sum_{j \ge 0} 
   \frac{(w d)^j}{j!} 
   + \sum_{n \ge 1} \sum_{d \ge 1} d^{t-1} q^{nd}  \\
&=&-\sum_{j \ge 1} \frac{w^j}{j!}
   \sum_{n \ge 1} \sum_{d \ge 1} d^{t+j-1}q^{nd}   \\
&=&-\sum_{j \ge 1} \frac{(t+j-1)!}{j!} [t+j] w^j
\end{eqnarray*} 
by \eqref{def_qMZV.04}. 
By \eqref{BK3-2.4}, $D_x^tf|_{x=1} \in \qMZV[[w]] \cdot w$ 
if $t \ge 1$, and the coefficient of $w^j$ in $D_x^tf|_{x=1}$ 
has weight $t+j$. Since $\qMZV$ is a $\Q$-algebra, 
\eqref{AAA06250955am.1} implies 
$$
a_{m} \in \qMZV[[w]]\cdot w
$$
if $m \ge 1$. Moreover, the coefficient of $w^n$ in $a_m$ 
has weight 
$
\sum_{t \ge 1} tk_t+n = m+n. 
$
Since $a_0 = 1$, we conclude that 
$$
\frac{(q)_\infty (xyq)_\infty}{(xq)_\infty (yq)_\infty} 
= 1 + \sum_{m \ge 1} a_m z^m
\in 1 + \qMZV[[z, w]] \cdot zw. 
$$
In addition, for $m,n \ge 1$, the coefficient of $z^m w^n$ has weight $m+n$.
\end{proof}

\begin{remark}   \label{AAA07020913am}
Let $x = e^{z}$ and $y = e^{w}$.
We see from the proof of Lemma~\ref{AAA06250955am} that
\begin{eqnarray*}     
& &\frac{(q)_\infty (xyq)_\infty}{(xq)_\infty (yq)_\infty}   \\
&=&1 + \sum_{m \ge 1} z^m \cdot \sum_{k_1+2k_2+3k_3+\cdots = m}
  \frac{1}{k_1! k_2! k_3! \cdots} \prod_{t \ge 1}
  \left (-\sum_{j \ge 1} \frac{(t+j-1)!}{t! \cdot j!} [t+j] 
  w^j\right )^{k_t}    \\
&=&1 - [2]zw - [3] (z + w)zw - [4]z^3 w 
   + \left (\frac12 [2]^2 - \frac32 [4] \right ) z^2 w^2 
   - [4]z w^3 + O(5)
\end{eqnarray*}
where $O(5)$ denotes the terms with the combined degree of $z$ and $w$ 
being at least $5$. Moreover, for $m \ge 2$, 
the coefficients of $z^{m-1} w$ and $zw^{m-1}$ are equal to $-[m]$.
\end{remark}

\begin{theorem}   \label{AAA06261252pm}
Let $(a)_\infty = (a; q)_\infty = (1-a)(1-aq)(1-aq^2) \cdots$.
Fix $r \ge 1$. For $1 \le j \le r$, let $x_j = e^{z_j}$ 
and $y_j = e^{w_j}$. If $x_1 \cdots x_r = y_1 \cdots y_r$,
then 
\begin{eqnarray}    \label{AAA06261252pm.0}
\prod_{j = 1}^r \frac{(x_jq)_\infty}{(y_jq)_\infty} 
\in \qMZV[[z_1, \ldots, z_r, w_1, \ldots, w_{r-1}]],
\end{eqnarray}
the constant term is $1$, and the coefficient of 
$z_1^{m_1} \cdots z_r^{m_r} w_1^{n_1} \cdots w_{r-1}^{n_{r-1}}$ has weight
$$
m_1 + \ldots + m_r + n_1 + \ldots + n_{r-1}.
$$
\end{theorem}
\begin{proof}
First of all, since $x_1 \cdots x_r = y_1 \cdots y_r$, we get 
$$
w_r = z_1 + \ldots + z_r - w_1 - \ldots - w_{r-1}.
$$

Next, it is clear that the constant term of the product in 
\eqref{AAA06261252pm.0} is $1$. 
Moreover, the case $r=1$ is trivially true. 
For $r \ge 2$, we can rewrite the product in 
\eqref{AAA06261252pm.0} as
$$
  \prod_{j = 1}^r \frac{(x_jq)_\infty}{(y_jq)_\infty} 
= \left (\frac{(q)_\infty(x_1x_2q)_\infty}{(x_1q)_\infty (x_2q)_\infty} 
  \right )^{-1} \cdot
  \frac{(q)_\infty (y_1y_2q)_\infty}{(y_1q)_\infty(y_2q)_\infty} \cdot
  \left (\frac{(x_1x_2q)_\infty}{(y_1y_2q)_\infty}
  \prod_{j = 3}^r \frac{(x_jq)_\infty}{(y_jq)_\infty} \right ).
$$
Therefore, \eqref{AAA06261252pm.0} and the statement about the weights 
follow from Lemma~\ref{AAA06250955am}, the fact that $\qMZV$ is 
a $\Q$-algebra, and induction on $r$.
\end{proof}

\section{\bf The basic case $\mathfrak P_{2}^{0,0}$} 
\label{sect_1to2}

Recall the trace $\mathfrak P_{N}^{a,b}$ from Definition~\ref{def-zmvn}.
In this section, we will verify Theorem~\ref{Intro-yij1toNm>0} 
for the basic case 
\begin{eqnarray*}     
   \mathfrak P_{2}^{0,0}
&=&\frac{(x_1x_2y)_\infty (y)_\infty}{(x_1y)_\infty (x_2y)_\infty} 
   \frac{(u_1u_2y^{-1}q)_\infty (y^{-1}q)_\infty}
     {(u_1y^{-1}q)_\infty (u_2y^{-1}q)_\infty}     \\
&\in&\Q[[q, y, y^{-1}, z_1, z_2, v_1, v_2]]
\end{eqnarray*}   
where $x_j = e^{z_j}$ and $u_j = e^{v_j}$ for $j \in \{1, 2\}$. 
The purpose of presenting this basic case first is to illustrate 
the main ideas in the next section where we prove 
Theorem~\ref{Intro-yij1toNm>0} (= Theorem~\ref{yij1toNm>0}) 
in full generality. 

The following lemma regarding the sums of powers should be well-known 
to the experts in analytic number theory. However, the author is unable 
to find a reference. For completeness, we present its proof here.

\begin{lemma}  \label{sumofpowers} 
Fix integers $n \ge 0$, $r \ge 1$ and $t_1, \ldots, t_{r} \ge 1$. Then, 
\begin{enumerate}
\item[{\rm (i)}]
$\displaystyle{\sum_{d_1, \ldots, d_r \ge 1, \, d_1+ \ldots + d_r \le n}
  \prod_{i=1}^{r} d_i^{t_i}}$
is a polynomial in $\Q[n] \cdot n$ of degree $\sum_{i=1}^{r} t_i+r$.

\item[{\rm (ii)}]
$\displaystyle{\sum_{d_1, \ldots, d_r \ge 1, \, d_1+ \ldots + d_r = n}
  \prod_{i=1}^{r} d_i^{t_i}}$
is a polynomial in $\Q[n] \cdot n$ of degree $\sum_{i=1}^{r} t_i+r-1$.
\end{enumerate}
\end{lemma}
\begin{proof}
(i) We use induction on $r$. Denote 
$\displaystyle{\sum_{d_1, \ldots, d_r \ge 1, \, d_1+ \ldots + d_r \le n}
\prod_{i=1}^{r} d_i^{t_i}}$ by $A_r$. Then
$$
A_r = \sum_{\substack{d_1, \ldots, d_r \ge 0, \, d_1+ \ldots + d_r \le n}}
  \prod_{i=1}^{r} d_i^{t_i}
$$
since $t_1, \ldots, t_{r} \ge 1$. When $r = 1$, the lemma is 
the classical Faulhaber's formula. Assume our lemma holds for $r-1$. 
We have
$$
A_r = \sum_{\substack{0 \le d_r \le n}} d_r^{t_r} 
  \sum_{\substack{d_1, \ldots, d_{r-1} \ge 0, \, d_1+ \ldots + d_{r-1} 
  \le n-d_r}} \prod_{i=1}^{r-1} d_i^{t_i}.
$$
By induction, $\sum_{\substack{d_1, \ldots, d_{r-1} \ge 0, 
\, d_1+ \ldots + d_{r-1} \le n-d_r}} \prod_{i=1}^{r-1} d_i^{t_i}$ 
is a polynomial in 
$$
\Q[(n-d_r)] \cdot (n-d_r)
$$ 
of degree $\sum_{i=1}^{r-1} (t_i+1)$. 
So $A_r$ is a $\Q$-linear combination of expressions of the form
$$
\sum_{\substack{0 \le d_r \le n}} d_r^{t_r} (n^a d_r^b)
$$
with $a, b \ge 0$ and $a+b \le \sum_{i=1}^{r-1} (t_i+1)$. 
By Faulhaber's formula,
$$
  \sum_{\substack{0 \le d_r \le n}} d_r^{t_r} (n^a d_r^b)
= n^a \cdot \sum_{\substack{0 \le d_r \le n}} d_r^{t_r+b}
$$
is a polynomial in $\Q[n] \cdot n^{a+1}$ of degree $a + (t_r+b+1) \le 
\sum_{i=1}^{r} (t_i+1)$. It follows that $A_r$ is a polynomial in 
$\Q[n] \cdot n$ of degree $\sum_{i=1}^{r} t_i+r$.

(ii) Denote $\displaystyle{\sum_{d_1, \ldots, d_r \ge 1, 
\, d_1+ \ldots + d_r = n} \prod_{i=1}^{r} d_i^{t_i}}$ by $B_r$.
Then, 
\begin{eqnarray*}    
   B_r 
&=&\sum_{d_1, \ldots, d_r \ge 0, \, d_1+ \ldots + d_r = n}
   \prod_{i=1}^{r} d_i^{t_i}   \\
&=&\sum_{d_1, \ldots, d_{r-1} \ge 0, \, d_1+ \ldots + d_{r-1} \le n}
\prod_{i=1}^{r-1} d_i^{t_i} \cdot (n - d_1 - \ldots - d_{r-1})^{t_r}.
\end{eqnarray*}
Thus, $B_r$ is a $\Q$-linear combination of expressions of the form
$$
n^a \cdot \sum_{d_1, \ldots, d_{r-1} \ge 0, \, d_1+ \ldots + d_{r-1} \le n}
\prod_{i=1}^{r-1} d_i^{\w t_i}
$$
where $a \ge 0$, $\w t_i \ge 1$ and $a + \sum_{i=1}^{r-1} \w t_i 
= \sum_{i=1}^{r} t_i$. By (i), $B_r$ is a polynomial in 
$\Q[n] \cdot n$ of degree $a + \sum_{i=1}^{r-1} \w t_i + r - 1
= \sum_{i=1}^{r} t_i+r-1$.
\end{proof}

Our next lemma shows that certain summation of the form 
\begin{eqnarray}  \label{summation.1}
\sum_{\substack{n_1, \ldots, n_r \ge 0, d_1, \ldots, d_r \ge 1\\
    n_{r+1}, \ldots, n_{r+s}, d_{r+1}, \ldots, d_{r+s} \ge 1\\
    d_1+ \ldots + d_r = d_{r+1} + \ldots + d_{r+s}}}
\end{eqnarray}
can be written as a finite $\Q$-linear combination of expressions of the form
\begin{eqnarray}   \label{summation.2}
\sum_{\substack{\w n_1 > \ldots > \w n_f \ge 0, 
    \w d_1, \ldots, \w d_f \ge 1\\
    \w n_{f+1}> \ldots > \w n_{f+g} \ge 1, 
    \w d_{f+1}, \ldots, \w d_{f+g} \ge 1\\
    \w d_1+ \ldots + \w d_f = \w d_{f+1} + \ldots + \w d_{f+g}}}.
\end{eqnarray}
Note that in \eqref{summation.1}, there is no relation among 
$n_1, \ldots, n_r \ge 0$ (respectively, among 
$n_{r+1}, \ldots, n_{r+s} \ge 1$). 
However, in \eqref{summation.2}, we impose the relations 
$\w n_1 > \ldots > \w n_f \ge 0$ (respectively,  
$\w n_{f+1}> \ldots > \w n_{f+g} \ge 1$).

\begin{lemma}  \label{finite_redu} 
Fix integers $r, s \ge 1$ and $u_1, \ldots, u_{r+s}, 
t_1, \ldots, t_{r+s} \ge 0$. Then, 
\begin{eqnarray}  \label{finite_redu.01}
\sum_{\substack{n_1, \ldots, n_r \ge 0, d_1, \ldots, d_r \ge 1\\
    n_{r+1}, \ldots, n_{r+s}, d_{r+1}, \ldots, d_{r+s} \ge 1\\
    d_1+ \ldots + d_r = d_{r+1} + \ldots + d_{r+s}}}
  \prod_{i=1}^{r+s} n_i^{u_i}d_i^{t_i} \cdot q^{\sum_{i=1}^{r+s} n_i d_i}
\end{eqnarray}
is a finite $\Q$-linear combination of expressions of the form
\begin{eqnarray}  \label{finite_redu.02}
\sum_{\substack{\w n_1 > \ldots > \w n_f \ge 0, 
    \w d_1, \ldots, \w d_f \ge 1\\
    \w n_{f+1}> \ldots > \w n_{f+g} \ge 1, 
    \w d_{f+1}, \ldots, \w d_{f+g} \ge 1\\
    \w d_1+ \ldots + \w d_f = \w d_{f+1} + \ldots + \w d_{f+g}}}
    \prod_{i=1}^{f+g} \w n_i^{\w u_i} \w d_i^{\w t_i} 
    \cdot q^{\sum_{i=1}^{f+g} \w n_i \w d_i}
\end{eqnarray}
where $r, s \ge 1$, $\w u_i, \w t_i \ge 0$ and 
$\sum_{i=1}^{f+g} (\w u_i + \w t_i + 1)  
\le \sum_{i=1}^{r+s} (s_i+t_i + 1)$. Moreover,
if $t_1, \ldots, t_{r+s} \ge 1$, then $\w t_1, \ldots, \w t_{f+g} \ge 1$. 
If $u_1, \ldots, u_{r+s} \ge 1$, 
then $\w u_1, \ldots, \w u_{f+g} \ge 1$.
\end{lemma}
\begin{proof}
Denote \eqref{finite_redu.01} by $W$. Then, $W$ can be decomposed as 
the finite sum of expressions of {\it the form \eqref{finite_redu.01} 
together with the following conditions}
$$
n_{1} = \ldots = n_{r_1} > n_{r_1+1} = \ldots = n_{r_1+r_2} > \cdots 
> n_{r_1+\ldots + r_{f-1} +1} = \ldots = n_{r_1+\ldots + r_{f}},
$$
$$
n_{r+1} = \ldots = n_{r+s_1} > 
\cdots > n_{r+s_1+\ldots + s_{g-1} +1} = \ldots = n_{r+ s_1+\ldots + s_g}
$$ 
where $f, g, r_i, s_i \ge 1$, $r_1+\ldots + r_{f} = r$ and 
$s_1+\ldots + s_g = s$. Put 
$$
\w u_j = \sum_{i=1}^{r_j} u_{r_1+\ldots+r_{j-1}+i}
$$ 
for $1 \le j \le f$ and 
$\w u_{f+j} 
= \sum_{i=1}^{s_j} u_{r+s_1+\ldots+s_{j-1}+i}$ for $1 \le j \le g$.
Then we see that $W$ is the finite sum of expressions of the form
$$
\sum_{\substack{\w n_1 > \ldots > \w n_f \ge 0, 
    \w d_1, \ldots, \w d_f \ge 1\\
    \w n_{f+1}> \ldots > \w n_{f+g} \ge 1, 
    \w d_{f+1}, \ldots, \w d_{f+g} \ge 1\\
    \w d_1+ \ldots + \w d_f = \w d_{f+1} + \ldots + \w d_{f+g}}}
    \prod_{j=1}^{f} \w n_j^{\w u_j} \left (\sum_{d_{r_1+\ldots+r_{j-1}+1} 
    + \ldots + d_{{r_1+\ldots+r_{j}}} = \w d_j} \prod_{i=1}^{r_j} 
    d_{r_1+\ldots+r_{j-1}+i}^{t_{r_1+\ldots+r_{j-1}+i}} \right )
$$
$$
\cdot \prod_{j=1}^{g} \w n_{f+j}^{\w u_{f+j}} 
    \left (\sum_{d_{r+s_1+\ldots+s_{j-1}+1} 
    + \ldots + d_{{r+s_1+\ldots+s_{j}}} = \w d_j} \prod_{i=1}^{s_j} 
    d_{r+s_1+\ldots+s_{j-1}+i}^{t_{r+s_1+\ldots+s_{j-1}+i}} \right )
    \cdot q^{\sum_{i=1}^{f+g} \w n_i \w d_i}.
$$
By Lemma~\ref{sumofpowers}~(ii), $W$ 
is a finite $\Q$-linear combination of expressions of the form
$$
\sum_{\substack{\w n_1 > \ldots > \w n_f \ge 0, 
    \w d_1, \ldots, \w d_f \ge 1\\
    \w n_{f+1}> \ldots > \w n_{f+g} \ge 1, 
    \w d_{f+1}, \ldots, \w d_{f+g} \ge 1\\
    \w d_1+ \ldots + \w d_f = \w d_{f+1} + \ldots + \w d_{f+g}}}
    \prod_{i=1}^{f+g} \w n_i^{\w u_i} \w d_i^{\w t_i} 
    \cdot q^{\sum_{i=1}^{f+g} \w n_i \w d_i}.
$$
where $\w u_i, \w t_i \ge 0$ and $\sum_{i=1}^{f+g} (\w u_i + \w t_i)$ 
is at most equal to
\begin{eqnarray*}
& &\sum_{i=1}^{r+s} u_i 
   + \sum_{j=1}^f \left (\sum_{i=1}^{r_j} t_{r_1+\ldots+
       r_{j-1}+i} + r_j - 1 \right )   \\
& &+ \sum_{j=1}^g \left (\sum_{i=1}^{s_j} t_{r+s_1+\ldots+
       s_{j-1}+i} + s_j - 1 \right )   \\
&=&\sum_{i=1}^{r+s} (s_i+t_i) + r+s -f-g. 
\end{eqnarray*}
Thus, $\sum_{i=1}^{f+g} (\w u_i + \w t_i +1) 
\le \sum_{i=1}^{r+s} (s_i+t_i + 1)$.
Note that the coefficients in the linear combination depend only on 
$r,s$ and $t_1, \ldots, t_{r+s}$. Moreover,
if $t_1, \ldots, t_{r+s} \ge 1$, then $\w t_1, \ldots, \w t_{f+g} \ge 1$. 
Similarly, if $u_1, \ldots, u_{r+s} \ge 1$, 
then $\w u_1, \ldots, \w u_{f+g} \ge 1$.
\end{proof}

In the lemma below, we prove that the expression \eqref{finite_redu.02} 
is contained in $\BD$.

\begin{lemma}  \label{lemma_t2t1yq} 
Fix integers $r, s \ge 1$ and $u_1, \ldots, u_{r+s}, 
t_1, \ldots, t_{r+s} \ge 0$. Then,
\begin{eqnarray}  \label{lemma_t2t1yq.0}
\sum_{\substack{n_1 > \ldots > n_r \ge 0, d_1, \ldots, d_r \ge 1\\
    n_{r+1} > \ldots > n_{r+s} \ge 1, d_{r+1}, \ldots, d_{r+s} \ge 1\\
    d_1+ \ldots + d_r = d_{r+1} + \ldots + d_{r+s}}}
  \prod_{i=1}^{r+s} n_i^{u_i}d_i^{t_i} \cdot q^{\sum_{i=1}^{r+s} n_i d_i}
\quad \in \quad \BD.
\end{eqnarray}
More precisely, the above summation is a finite $\Q$-linear combination of 
bi-brackets 
$$
\begin{bmatrix} t_1', \ldots, t_\ell'\\u_1', \ldots, u_\ell' \end{bmatrix}
$$
of weights at most $\sum_{i=1}^{r+s} (u_i + t_i+1)$. In addition, 
if $t_1, \ldots, t_{r+s} \ge 1$, then $t_1' \ge 2$. 
\end{lemma}
\begin{proof}
Use induction on $s$. 
Denote the summation in \eqref{lemma_t2t1yq.0} by $W_s$. 
For simplicity, we have suppressed the dependence of $W_s$ on 
$r, u_1, \ldots, u_{r+s}, t_1, \ldots, t_{r+s}$. 
The idea of the proof is to eliminate the parameters 
$d_{r+1}, \ldots, d_{r+s}$, via induction on $s$, from the condition 
$d_1+ \ldots + d_r = d_{r+1} + \ldots + d_{r+s}$ until it becomes void:
\begin{eqnarray}  \label{lemma_t2t1yq.100}
\w d_1+ \ldots + \w d_{\w r} = \w d_{\w r+1} + \ldots + \w d_{\w r+0}
\end{eqnarray}
($\w r$ may be different from $r$), 
which corresponds to a summation of the form $W_0$. 
During the elimination process, we may produce elements in $\BD$.

First of all, we rewrite $W_s$. Since $d_1+ \ldots + d_r 
- d_{r+1} - \ldots - d_{r+s-1} = d_{r+s} \ge 1$, 
$$
W_s = \sum_{\substack{n_1 > \ldots > n_r \ge 0, d_1, \ldots, d_r \ge 1\\
    n_{r+1} > \ldots > n_{r+s} \ge 1, d_{r+1}, \ldots, d_{r+s-1} \ge 1\\
    d_1+ \ldots + d_r > d_{r+1} + \ldots + d_{r+s-1}}}
  \prod_{i=1}^{r+s-1} n_i^{u_i}d_i^{t_i} \cdot n_{r+s}^{u_{r+s}}
  \left (\sum_{i=1}^r d_i - \sum_{i=1}^{s-1} d_{r+i} \right )^{t_{r+s}} 
$$
$$
\cdot q^{\sum_{i=1}^{r} (n_i + n_{r+s})d_i}
\cdot q^{\sum_{i=r+1}^{r+s-1} (n_i - n_{r+s})d_i}.
$$
So $W_s$ is a finite $\Q$-linear combination of expressions of the form
$$
\sum_{\substack{n_1 > \ldots > n_r \ge 0, d_1, \ldots, d_r \ge 1\\
  n_{r+1} > \ldots > n_{r+s} \ge 1, d_{r+1}, \ldots, d_{r+s-1} \ge 1\\
  d_1+ \ldots + d_r > d_{r+1} + \ldots + d_{r+s-1}}}
  \prod_{i=1}^{r+s-1} n_i^{u_i}d_i^{\w t_i} \cdot n_{r+s}^{u_{r+s}} 
  \cdot q^{\sum_{i=1}^{r} (n_i + n_{r+s})d_i}
  \cdot q^{\sum_{i=r+1}^{r+s-1} (n_i - n_{r+s})d_i}
$$
where $\w t_i \ge 0$ and $\sum_{i=1}^{r+s-1} \w t_i 
= \sum_{i=1}^{r+s} t_i$. 
After changes of variables (e.g, set $n = n_{r+s}$), 
$W_s$ is a finite $\Q$-linear combination of expressions of the form 
$$
\sum_{\substack{n_1 > \ldots > n_r \ge n \ge 1, d_1, \ldots, d_r \ge 1\\
    n_{r+1} > \ldots > n_{r+s-1} \ge 1, d_{r+1}, \ldots, d_{r+s-1} \ge 1\\
    d_1+ \ldots + d_r > d_{r+1} + \ldots + d_{r+s-1}}}
   \prod_{i=1}^{r} (n_i-n)^{u_i}d_i^{\w t_i} \cdot 
   \prod_{i=r+1}^{r+s-1} (n_i+n)^{u_i}d_i^{\w t_i} 
$$
$$
\cdot n^{u_{r+s}} \cdot q^{\sum_{i=1}^{r+s-1} n_i d_i}.
$$
which itself is a finite $\Q$-linear combination of expressions of the form
$$
\sum_{\substack{n_1 > \ldots > n_r \ge 1, d_1, \ldots, d_r \ge 1\\
    n_{r+1} > \ldots > n_{r+s-1} \ge 1, d_{r+1}, \ldots, d_{r+s-1} \ge 1\\
    d_1+ \ldots + d_r > d_{r+1} + \ldots + d_{r+s-1}}}
   \prod_{i=1}^{r+s-1} n_i^{\w u_i}d_i^{\w t_i}  
   \cdot q^{\sum_{i=1}^{r+s-1} n_i d_i} \cdot \sum_{n=1}^{n_r} n^a
$$
where $\w u_i, a \ge 0$ and $\sum_{i=1}^{r+s-1} \w u_i +a 
= \sum_{i=1}^{r+s} u_i$. By Lemma~\ref{sumofpowers}~(i), 
$W_s$ is a finite $\Q$-linear combination of expressions of the form
\begin{eqnarray}  \label{lemma_t2t1yq.1}
\sum_{\substack{n_1 > \ldots > n_r \ge 1, d_1, \ldots, d_r \ge 1\\
    n_{r+1} > \ldots > n_{r+s-1} \ge 1, d_{r+1}, \ldots, d_{r+s-1} \ge 1\\
    d_1+ \ldots + d_r > d_{r+1} + \ldots + d_{r+s-1}}}
   \prod_{i=1}^{r+s-1} n_i^{\hat u_i}d_i^{\w t_i}  
   \cdot q^{\sum_{i=1}^{r+s-1} n_i d_i}
\end{eqnarray}
where $\hat u_i, \w t_i \ge 0$, $\hat u_r \ge 1$ and 
$\sum_{i=1}^{r+s-1} \hat u_i \le 
\sum_{i=1}^{r+s-1} \w u_i +a + 1 = \sum_{i=1}^{r+s} u_i + 1$. So
\begin{eqnarray}  \label{lemma_t2t1yq.2}
\sum_{i=1}^{r+s-1} (\hat u_i + \w t_i + 1)   
\le \sum_{i=1}^{r+s} (u_i + t_i + 1).
\end{eqnarray}
Moreover,
if $t_1, \ldots, t_{r+s} \ge 1$, then $\w t_1, \ldots, \w t_{r+s-1} \ge 1$. 

When $s=1$, we see from \eqref{lemma_t2t1yq.1} and \eqref{lemma_t2t1yq.2}
that $W_1$ is a finite $\Q$-linear combination of expressions of the form
$$
\sum_{n_1 > \ldots > n_r \ge 1, d_1, \ldots, d_r \ge 1}
   \prod_{i=1}^{r} n_i^{\hat u_i}d_i^{\w t_i}  
   \cdot q^{\sum_{i=1}^{r} n_i d_i}
= \prod_{i=1}^{r} (\hat u_i! \cdot \w t_i!) \cdot 
   \begin{bmatrix} \w t_1 + 1, \ldots, \w t_r + 1\\
     \hat u_1, \ldots, \hat u_r 
   \end{bmatrix}
$$
where $\sum_{i=1}^{r} (\hat u_i + \w t_i + 1)   
\le \sum_{i=1}^{r+1} (u_i + t_i + 1)$. Moreover,
if $t_1, \ldots, t_{r+1} \ge 1$, then $\w t_1, \ldots, \w t_{r} \ge 1$. 
Hence the lemma holds for $s=1$.

Next, fix $s \ge 2$. Assume that the lemma holds for all the $W_\ell$'s 
(with various integers $r \ge 1, u_1, \ldots, u_{r+\ell}, 
t_1, \ldots, t_{r+\ell} \ge 0$) whenever $1 \le \ell \le s-1$. 
We will prove that the lemma holds for $W_s$ as well.
Rewrite \eqref{lemma_t2t1yq.1} as
\begin{eqnarray}  \label{lemma_t2t1yq.3}
\sum_{\substack{n_1 > \ldots > n_r \ge 1, d_1, \ldots, d_r \ge 1\\
    n_{r+1} > \ldots > n_{r+s-1} \ge 1, d_{r+1}, \ldots, d_{r+s-1} \ge 1}}
- \sum_{\substack{n_1 > \ldots > n_r \ge 1, d_1, \ldots, d_r \ge 1\\
    n_{r+1} > \ldots > n_{r+s-1} \ge 1, d_{r+1}, \ldots, d_{r+s-1} \ge 1\\
    d_1+ \ldots + d_r = d_{r+1} + \ldots + d_{r+s-1}}}
\end{eqnarray}
\begin{eqnarray}  \label{lemma_t2t1yq.4}
- \sum_{\substack{n_1 > \ldots > n_r \ge 1, d_1, \ldots, d_r \ge 1\\
    n_{r+1} > \ldots > n_{r+s-1} \ge 1, d_{r+1}, \ldots, d_{r+s-1} \ge 1\\
    d_1+ \ldots + d_r < d_{r+1} + \ldots + d_{r+s-1}}}.
\end{eqnarray}
The first summation in \eqref{lemma_t2t1yq.3} is equal to
$$
\prod_{i=1}^{r+s-1} (\hat u_i! \cdot \w t_i!) \cdot 
   \begin{bmatrix} \w t_1 + 1, \ldots, \w t_r + 1\\
     \hat u_1, \ldots, \hat u_r 
   \end{bmatrix}
\cdot \begin{bmatrix} \w t_{r+1} + 1, \ldots, \w t_{r+s-1} + 1\\
     \hat u_{r+1}, \ldots, \hat u_{r+s-1}
   \end{bmatrix}
\in \BD.
$$
Next, we deal with \eqref{lemma_t2t1yq.4}. 
Note that \eqref{lemma_t2t1yq.4} is equal to
$$
- \sum_{\substack{n_1 > \ldots > n_r \ge 1, d_1, \ldots, d_r, d \ge 1\\
    n_{r+1} > \ldots > n_{r+s-1} \ge 1, d_{r+1}, \ldots, d_{r+s-1} \ge 1\\
    d_1+ \ldots + d_r + d = d_{r+1} + \ldots + d_{r+s-1}}}
    \prod_{i=1}^{r+s-1} n_i^{\hat u_i}d_i^{\w t_i}  
    \cdot q^{\sum_{i=1}^{r+s-1} n_i d_i}.
$$
Since $d_1+ \ldots + d_r + d - d_{r+1} - \ldots - d_{r+s-2} 
= d_{r+s-1} \ge 1$, \eqref{lemma_t2t1yq.4} is equal to
$$
- \sum_{\substack{n_1 > \ldots > n_r \ge 1, d_1, \ldots, d_r, d \ge 1\\
    n_{r+1} > \ldots > n_{r+s-1} \ge 1, d_{r+1}, \ldots, d_{r+s-2} \ge 1\\
    d_1+ \ldots + d_r + d > d_{r+1} + \ldots + d_{r+s-2}}}
    \prod_{i=1}^{r+s-2} n_i^{\hat u_i}d_i^{\w t_i} \cdot 
    n_{r+s-1}^{\hat u_{r+s-1}} \left (\prod_{i=1}^{r} 
    d_i + d - \prod_{i=1}^{s-2} d_{r+i} \right )^{\w t_{r+s-1}}
$$
$$
\cdot q^{\sum_{i=1}^{r+s-2} n_i d_i}
\cdot q^{n_{r+s-1} (d_1+ \ldots + d_r + d - d_{r+1} - \ldots - d_{r+s-2})}
$$
which is a finite $\Q$-linear combination of expressions of the form
$$
\sum_{\substack{n_1 > \ldots > n_r > n \ge 1, d_1, \ldots, d_r, d \ge 1\\
    n_{r+1} > \ldots > n_{r+s-2} \ge 1, d_{r+1}, \ldots, d_{r+s-2} \ge 1\\
    d_1+ \ldots + d_r + d > d_{r+1} + \ldots + d_{r+s-2}}}
    \prod_{i=1}^{r} (n_i-n)^{\hat u_i}d_i^{\hat t_i} \cdot 
    \prod_{i=r+1}^{r+s-2} (n_i+n)^{\hat u_i}d_i^{\hat t_i} 
$$
$$
\cdot n^{\hat u_{r+s-1}} d^{\hat t_{r+s-1}}
\cdot q^{\sum_{i=1}^{r+s-2} n_i d_i} \cdot q^{nd}
$$
where $\hat t_i \ge 0$ ({\it it is possible that $\hat t_{r+s-1}=0$ 
even if $t_1, \ldots, t_{r+s} \ge 1$}) and 
$$
  \sum_{i=1}^{r+s-1} \hat t_i = \sum_{i=1}^{r+s-1} \w t_i 
= \sum_{i=1}^{r+s} t_i. 
$$
It follows that \eqref{lemma_t2t1yq.4} 
is a finite $\Q$-linear combination of expressions of the form
\begin{eqnarray}  \label{lemma_t2t1yq.6}
\sum_{\substack{n_1 > \ldots > n_r > n \ge 1, d_1, \ldots, d_r, d \ge 1\\
    n_{r+1} > \ldots > n_{r+s-2} \ge 1, d_{r+1}, \ldots, d_{r+s-2} \ge 1\\
    d_1+ \ldots + d_r + d > d_{r+1} + \ldots + d_{r+s-2}}}
    \prod_{i=1}^{r+s-2} n_i^{\overline u_i}d_i^{\hat t_i} \cdot 
    n^{\overline u_{r+s-1}} d^{\hat t_{r+s-1}}
    \cdot q^{\sum_{i=1}^{r+s-2} n_i d_i + nd}
\end{eqnarray} 
where $\overline u_i \ge 0$ and $\sum_{i=1}^{r+s-1} \overline u_i 
= \sum_{i=1}^{r+s-1} \hat u_i$. Note that 
$$
    \sum_{i=1}^{r+s-1} (\overline u_i + \hat t_i + 1)  
= \sum_{i=1}^{r+s-1} (\hat u_i + \w t_i + 1)    
\le \sum_{i=1}^{r+s} (u_i + t_i + 1)
$$
by \eqref{lemma_t2t1yq.2}, and that if $t_1, \ldots, t_{r+1} \ge 1$, 
then $\hat t_1, \ldots, \hat t_{r} \ge 1$. Moreover, 
\eqref{lemma_t2t1yq.6} is of the form \eqref{lemma_t2t1yq.1} 
with the index $s$ there being replaced by $s-1$,
and we repeat the above argument beginning at line \eqref{lemma_t2t1yq.3}.

Now we deal with the second summation in \eqref{lemma_t2t1yq.3} 
which is equal to
\begin{eqnarray}  \label{lemma_t2t1yq.5}
\sum_{\substack{n_1 > \ldots > n_r \ge 0, d_1, \ldots, d_r \ge 1\\
    n_{r+1} > \ldots > n_{r+s-1} \ge 1, d_{r+1}, \ldots, d_{r+s-1} \ge 1\\
    d_1+ \ldots + d_r = d_{r+1} + \ldots + d_{r+s-1}}} 
    \prod_{i=1}^{r+s-1} n_i^{\hat u_i}d_i^{\w t_i}  
    \cdot q^{\sum_{i=1}^{r+s-1} n_i d_i}
\end{eqnarray}
$$
- 0^a \cdot 
  \sum_{\substack{n_1 > \ldots > n_{r-1} \ge 1, d_1, \ldots, d_{r-1} \ge 1\\
    n_{r} > \ldots > n_{r+s-2} \ge 1, d_{r}, \ldots, d_{r+s-2} \ge 1\\
    d_1+ \ldots + d_{r-1} < d_{r} + \ldots + d_{r+s-2}}}
    \prod_{i=1}^{r+s-2} n_i^{\overline u_i} d_i^{\overline t_i}  
    \cdot q^{\sum_{i=1}^{r+s-2} n_i d_i} \cdot  
    \left (\sum_{i=0}^{s-2} d_{r+i} - \sum_{i=1}^{r-1} d_i \right )^b
$$
where $a = \hat u_r \ge 0$ (we set $0^0 =1$), $b = \w t_r$, 
we have renamed $\hat u_i$ and $\w t_i$ as 
$\overline u_i$ and $\overline t_i$ 
respectively when $1 \le i \le r-1$, and we have renamed 
$n_{r+i}, d_{r+i}, \hat u_{r+i}, \w t_{r+i}$ as $n_{r+i-1}, d_{r+i-1}, 
\overline u_{r+i-1}, \overline t_{r+i-1}$ respectively 
when $1 \le i \le s-1$. So 
$$
    \sum_{i=1}^{r+s-2} (\overline u_i + \overline t_i + 1) + b 
\le \sum_{i=1}^{r+s-1} (\hat u_i + \w t_i + 1) - 1
< \sum_{i=1}^{r+s} (u_i + t_i + 1)
$$
by \eqref{lemma_t2t1yq.2}. Since line \eqref{lemma_t2t1yq.5} is 
of the form $W_{s-1}$, we deduce from induction that 
line \eqref{lemma_t2t1yq.5} is 
a finite $\Q$-linear combination of bi-brackets of weights at most 
$$
  \sum_{i=1}^{r+s-1} (\hat u_i + \w t_i + 1)
\le \sum_{i=1}^{r+s} (u_i + t_i + 1).
$$
The line below line \eqref{lemma_t2t1yq.5} is a finite $\Q$-linear combination of the expressions 
$$
0^a \cdot 
\sum_{\substack{n_1 > \ldots > n_{r-1} \ge 1, d_1, \ldots, d_{r-1} \ge 1\\
    n_{r} > \ldots > n_{r+s-2} \ge 1, d_{r}, \ldots, d_{r+s-2} \ge 1\\
    d_1+ \ldots + d_{r-1} < d_{r} + \ldots + d_{r+s-2}}}
    \prod_{i=1}^{r+s-2} n_i^{\overline u_i} d_i^{\hat t_i}  
    \cdot q^{\sum_{i=1}^{r+s-2} n_i d_i}
$$
where $\hat t_i \ge 0$ and $\sum_{i=1}^{r+s-2} \hat t_i 
= \sum_{i=1}^{r+s-2} t_i + b$. The preceding line is either $0$ or 
is of the form \eqref{lemma_t2t1yq.4} with the index $r$ there 
being replaced by $r-1$, and so we repeat the above argument for 
\eqref{lemma_t2t1yq.4}.

In summary, we conclude that $W_s$ is a finite $\Q$-linear combination of 
bi-brackets 
$$
\begin{bmatrix} t_1', \ldots, t_\ell'\\u_1', \ldots, u_\ell' \end{bmatrix}
$$
of weights at most $\sum_{i=1}^{r+s} (u_i + t_i+1)$. In addition, 
if $t_1, \ldots, t_{r+s} \ge 1$, then $t_1' \ge 2$. 
\end{proof}

\begin{remark}   \label{rmk_rs}
In the proof of Lemma~\ref{lemma_t2t1yq}, instead of eliminating 
$d_{r+1}, \ldots, d_{r+s}$ via induction on $s$, 
we may eliminate the parameters $d_{1}, \ldots, d_{r}$, 
via induction on $r$, from the condition 
$d_1+ \ldots + d_r = d_{r+1} + \ldots + d_{r+s}$ until it becomes void:
\begin{eqnarray}  \label{rmk_rs.100}
\w d_1+ \ldots + \w d_{0} = \w d_{1} + \ldots + \w d_{\w s}
\end{eqnarray}
($\w s$ may be different from $s$).
In the following, we sketch the process. To avoid confusion, 
we denote the summation in \eqref{lemma_t2t1yq.0} by $V_r$ (so $V_r$ is 
the same as the $W_s$ in the proof of Lemma~\ref{lemma_t2t1yq}).
As in \eqref{lemma_t2t1yq.1}, 
$V_r$ is a finite $\Q$-linear combination of expressions of the form
\begin{eqnarray}  \label{rmk_rs.1}
\sum_{\substack{n_1 > \ldots > n_{r-1} \ge 1, d_1, \ldots, d_{r-1} \ge 1\\
    n_{r} > \ldots > n_{r+s-1} \ge 1, d_{r}, \ldots, d_{r+s-1} \ge 1\\
    d_1+ \ldots + d_{r-1} < d_{r} + \ldots + d_{r+s-1}}}
   \prod_{i=1}^{r+s-1} n_i^{\hat u_i}d_i^{\w t_i}  
   \cdot q^{\sum_{i=1}^{r+s-1} n_i d_i}.
\end{eqnarray}
So Lemma~\ref{lemma_t2t1yq} holds for $r = 1$. Next, fix $r \ge 2$.
Assume that the lemma holds for all the $W_\ell$'s whenever 
$1 \le \ell \le r-1$. Rewrite \eqref{rmk_rs.1} as
\begin{eqnarray}  \label{rmk_rs.2}
\sum_{\substack{n_1 > \ldots > n_{r-1} \ge 1, d_1, \ldots, d_{r-1} \ge 1\\
    n_{r} > \ldots > n_{r+s-1} \ge 1, d_{r}, \ldots, d_{r+s-1} \ge 1}}
- \sum_{\substack{n_1 > \ldots > n_{r-1} \ge 1, d_1, \ldots, d_{r-1} \ge 1\\
    n_{r} > \ldots > n_{r+s-1} \ge 1, d_{r}, \ldots, d_{r+s-1} \ge 1\\
    d_1+ \ldots + d_{r-1} = d_{r} + \ldots + d_{r+s-1}}}
\end{eqnarray}
\begin{eqnarray}  \label{rmk_rs.3}
- \sum_{\substack{n_1 > \ldots > n_{r-1} \ge 1, d_1, \ldots, d_{r-1} \ge 1\\
    n_{r} > \ldots > n_{r+s-1} \ge 1, d_{r}, \ldots, d_{r+s-1} \ge 1\\
    d_1+ \ldots + d_{r-1} > d_{r} + \ldots + d_{r+s-1}}}.
\end{eqnarray}
The first summation in \eqref{rmk_rs.2} is contained in $\BD$, 
and \eqref{rmk_rs.3} can be reduced to \eqref{rmk_rs.1} with 
the index $r$ there replaced by $(r-1)$. 
The second summation in \eqref{rmk_rs.2} can be reduced to $V_{r-1}$ 
and \eqref{rmk_rs.3} with the index $s$ there replaced by $(s-1)$. 
\end{remark}

\begin{proposition}  \label{Prop_t2t1yq} 
Fix integers $r, s \ge 1$ and $u_1, \ldots, u_{r+s}, 
t_1, \ldots, t_{r+s} \ge 0$. Then, 
\begin{eqnarray}  \label{Prop_t2t1yq.0}
\sum_{\substack{n_1, \ldots, n_r \ge 0, d_1, \ldots, d_r \ge 1\\
    n_{r+1}, \ldots, n_{r+s}, d_{r+1}, \ldots, d_{r+s} \ge 1\\
    d_1+ \ldots + d_r = d_{r+1} + \ldots + d_{r+s}}}
  \prod_{i=1}^{r+s} n_i^{u_i}d_i^{t_i} \cdot q^{\sum_{i=1}^{r+s} n_i d_i}
\quad \in \quad \BD.
\end{eqnarray}
More precisely, the above summation is a finite $\Q$-linear combination of 
bi-brackets
$$
\begin{bmatrix} t_1', \ldots, t_\ell'\\u_1', \ldots, u_\ell' \end{bmatrix}
$$
of weights at most $\sum_{i=1}^{r+s} (u_i + t_i+1)$. In addition, 
if $t_1, \ldots, t_{r+s} \ge 1$, then $t_1' \ge 2$. 
\end{proposition}
\begin{proof}
Follows immediately from Lemma~\ref{finite_redu} and 
Lemma~\ref{lemma_t2t1yq}.
\end{proof}

Now we are ready to prove the basic case $\mathfrak P_{2}^{0,0}$ in 
Theorem~\ref{Intro-yij1toNm>0} .

\begin{theorem}  \label{t2t1yq} 
For $j \in \{1, 2\}$, let $x_j = e^{z_j}$ and $u_j = e^{v_j}$. Then,
\begin{eqnarray}    \label{t2t1yq.0}
\Coe_{y=0} \left (\frac{(x_1x_2y)_\infty (y)_\infty}
     {(x_1y)_\infty (x_2y)_\infty} 
\frac{(u_1u_2y^{-1}q)_\infty (y^{-1}q)_\infty}
     {(u_1y^{-1}q)_\infty (u_2y^{-1}q)_\infty} \right )
\in \qBD[[z_1,z_2,v_1,v_2]].
\end{eqnarray}
More precisely, the coefficient of $z_1^m z_2^n v_1^{\w m} v_2^{\w n}$ 
in the above expression is a finite $\Q$-linear combination of 
bi-brackets in $\qBD$ of weights at most $(m + n + \w m + \w n)$.
\end{theorem}
\begin{proof}
Regard $(x_1x_2y)_\infty (y)_\infty (x_1y)_\infty^{-1} (x_2y)_\infty^{-1}$ 
as an element in $\Q[[z_1,z_2,q,y]]$. Put
$$
f(z_1) = \ln\big ((x_1x_2y)_\infty (y)_\infty (x_1y)_\infty^{-1} 
(x_2y)_\infty^{-1} \big ).
$$
Then, $f(0) = 0$. For $m \ge 0$, let 
$$
a_{m} = \Coe_{z_1^m} 
\big ((x_1x_2y)_\infty (y)_\infty (x_1y)_\infty^{-1} (x_2y)_\infty^{-1} 
\big ) \in \C[[z_2,q,y]].
$$
By \eqref{AAA06250955am.1}, we have
\begin{eqnarray}  \label{t2t1yq.1}
  a_{m}
= \sum_{k_1+2k_2+3k_3+\cdots = m}
  \frac{1}{k_1! k_2! k_3! \cdots} \prod_{t \ge 1}
  \left (\frac{D_{x_1}^t f}{t!} \Big |_{x_1=1}\right )^{k_t}. 
\end{eqnarray} 
Since $f(z) = \ln\big ((y)_\infty (x_2y)_\infty^{-1} \big ) 
+ \sum_{n \ge 0} \ln (1-x_1x_2yq^n) - \sum_{n \ge 0} \ln (1-x_1yq^n)$,
\begin{eqnarray*} 
   D_{x_1}f  
&=&-\sum_{n \ge 0} \frac{x_1x_2yq^n}{1-x_1x_2yq^n} 
+ \sum_{n \ge 0} \frac{x_1yq^n}{1-x_1yq^n} \\
&=&-\sum_{n \ge 0} \sum_{d \ge 1} (x_1x_2yq^n)^d 
   + \sum_{n \ge 0} \sum_{d \ge 1} (x_1yq^n)^d.
\end{eqnarray*} 
It follows that for $t \ge 1$, $D_{x_1}^tf|_{x_1=1}$ is equal to
\begin{eqnarray*} 
& &-\sum_{n \ge 0} \sum_{d \ge 1} d^{t-1} (x_2yq^n)^d 
   + \sum_{n \ge 0} \sum_{d \ge 1} d^{t-1} (yq^n)^{d} \\
&=&-\sum_{n \ge 0} \sum_{d \ge 1} d^{t-1}y^dq^{nd} \sum_{j \ge 0} 
   \frac{(dz_2)^j}{j!} 
   + \sum_{n \ge 0} \sum_{d \ge 1} d^{t-1} y^dq^{nd}  \\
&=&-\sum_{j \ge 1} \frac{z_2^j}{j!}
   \sum_{n \ge 0} \sum_{d \ge 1} d^{t+j-1} y^dq^{nd} \\
&\in&\Q[[z_2, q,y]] \cdot z_2
\end{eqnarray*} 
By \eqref{t2t1yq.1}, $a_{m} \in \Q[[z_2, q,y]] \cdot z_2$ if $m \ge 1$. 
Moreover, the coefficient of $z_2^n$ in $a_m$ is 
a finite $\Q$-linear combination of expressions of the form:
\begin{eqnarray}  \label{t2t1yq.2}
\prod_{i=1}^r \left (\sum_{n_i \ge 0} \sum_{d_i \ge 1} 
d_i^{t_i+j_i-1} y^{d_i}q^{n_i d_i} \right )
\end{eqnarray}
where $t_i$ and $j_i$ are fixed positive integers satisfying 
$\sum_{i=1}^r t_i = m$ and $\sum_{i=1}^r j_i = n$.
Since $a_0 = 1$, we conclude that 
\begin{eqnarray}  \label{t2t1yq.3}
  \frac{(x_1x_2y)_\infty (y)_\infty}{(x_1y)_\infty (x_2y)_\infty}
= 1 + \sum_{m \ge 1} a_m z_1^m
\in 1 + \Q[[z_1, z_2, q, y]] \cdot z_1z_2. 
\end{eqnarray}
In addition, for $m,n \ge 1$, the coefficient of $z_1^m z_2^n$ is 
a finite $\Q$-linear combination of expressions of 
the form \eqref{t2t1yq.2}.

Similarly, 
\begin{eqnarray}  \label{t2t1yq.4}
\frac{(u_1u_2y^{-1}q)_\infty (y^{-1}q)_\infty}
     {(u_1y^{-1}q)_\infty (u_2y^{-1}q)_\infty}
\in 1 + \Q[[v_1, v_2, q, y^{-1}]] \cdot v_1v_2,
\end{eqnarray}
and for $\w m, \w n \ge 1$, the coefficient of $v_1^{\w m} v_2^{\w n}$ is 
a finite $\Q$-linear combination of expressions of the form:
\begin{eqnarray}  \label{t2t1yq.5}
\prod_{i=1}^s \left (\sum_{n_i \ge 1} \sum_{d_i \ge 1} 
d_i^{t_i+j_i-1} (y^{-1})^{d_i}q^{n_i d_i} \right )
\end{eqnarray}
where $t_i$ and $j_i$ are fixed positive integers with
$\sum_{i=1}^s t_i = \w m$ and $\sum_{i=1}^s j_i = \w n$. Put 
$$
  C^{m,n}_{\w m, \w n} 
= \Coe_{z_1^m z_2^n v_1^{\w m} v_2^{\w n}}
  \left (\frac{(x_1x_2y)_\infty (y)_\infty}
     {(x_1y)_\infty (x_2y)_\infty} 
  \frac{(u_1u_2y^{-1}q)_\infty (y^{-1}q)_\infty}
     {(u_1y^{-1}q)_\infty (u_2y^{-1}q)_\infty} \right )
\in \Q[[q, y, y^{-1}]].
$$
By \eqref{t2t1yq.2}-\eqref{t2t1yq.5}, $C^{m,n}_{\w m, \w n}$ is 
a finite $\Q$-linear combination of expressions of the form:
\begin{eqnarray}  \label{t2t1yq.50}
& &\prod_{i=1}^r \left (\sum_{n_i \ge 0} \sum_{d_i \ge 1} 
    d_i^{t_i+j_i-1} y^{d_i}q^{n_i d_i} \right )
  \cdot \prod_{i=1}^{s} \left (\sum_{n_i \ge 1} \sum_{d_i \ge 1} 
    d_i^{t_i+j_i-1} (y^{-1})^{d_i}q^{n_i d_i} \right ) 
    \nonumber   \\
&=&\sum_{\substack{n_1, \ldots, n_r \ge 0, d_1, \ldots, d_r \ge 1\\
    n_{r+1}, \ldots, n_{r+s} \ge 1, d_{r+1}, \ldots, d_{r+s} \ge 1}}
  \prod_{i=1}^{r+s} d_i^{t_i+j_i-1} 
  \cdot y^{\sum_{i=1}^r d_i - \sum_{i=1}^{s} d_{r+i}}
  q^{\sum_{i=1}^{r+s} n_i d_i}
\end{eqnarray}
where $t_i, j_i$ are fixed positive integers satisfying 
$$
\sum_{i=1}^r t_i = m, \sum_{i=1}^r j_i = n,
\sum_{i=1}^{s} t_{r+i} = \w m, \sum_{i=1}^{s} j_{r+i} = \w n. 
$$
Thus, $\Coe_{y=0} C^{m,n}_{\w m, \w n}$ is 
a finite $\Q$-linear combination of expressions of the form:
\begin{eqnarray}  \label{t2t1yq.6}
\sum_{\substack{n_1, \ldots, n_r \ge 0, d_1, \ldots, d_r \ge 1\\
    n_{r+1}, \ldots, n_{r+s} \ge 1, d_{r+1}, \ldots, d_{r+s} \ge 1\\
    d_1 + \ldots + d_r = d_{r+1} + \ldots + d_{r+s}}}
  \prod_{i=1}^{r+s} d_i^{t_i+j_i-1} 
  \cdot q^{\sum_{i=1}^{r+s} n_i d_i}.
\end{eqnarray}
By Proposition~\ref{Prop_t2t1yq}, $\Coe_{y=0} C^{m,n}_{\w m, \w n}$ 
is a finite $\Q$-linear combination of 
bi-brackets in $\qBD$ of weights at most $(m + n + \w m + \w n)$. 
This proves the theorem.
\end{proof}

\section{\bf Proof of Theorem~\ref{Intro-yij1toNm>0}} 
\label{sect_mpositive}

In this section, we will prove Theorem~\ref{Intro-yij1toNm>0} 
(= Theorem~\ref{yij1toNm>0}) regarding the trace $\mathfrak P_{N}^{a,b}$ 
(see Definition~\ref{def-zmvn}), after having settled the basic case 
$\mathfrak P_{2}^{0,0}$ in Theorem~\ref{t2t1yq} in the previous section. 
Our four results here (Lemma~\ref{finite_redu1toNm>0}, 
Lemma~\ref{lemma_t2t1yq1toNm>0}, Proposition~\ref{Prop_t2t1yq1toNm>0} 
and Theorem~\ref{yij1toNm>0})
correspond to Lemma~\ref{finite_redu}, Lemma~\ref{lemma_t2t1yq}, 
Proposition~\ref{Prop_t2t1yq} and Theorem~\ref{t2t1yq} respectively. 

We start with a lemma which is a direct generalization of 
Lemma~\ref{finite_redu}.

\begin{lemma}  \label{finite_redu1toNm>0} 
Let $N \ge 2$. For $1 \le i < j \le N$, fix $r^{i,j}, s^{i,j} \ge 0$. 
For $1 \le i < j \le N$ and $1 \le k \le r^{i,j} + s^{i,j}$, 
fix $u_k^{i,j}, t_k^{i,j} \ge 0$. 
For $1 \le i \le N$, fix $r_i, s_i \ge 0$. For $1 \le i \le N$
and $1 \le k \le r_i + s_i$, fix $u_k^{(i)}, t_k^{(i)} \ge 0$. Then, 
$$
\prod_{1 \le i < j \le N}
\sum_{\substack{n_1^{i,j}, \ldots, n_{r^{i,j}}^{i,j} \ge 0, 
    d_1^{i,j}, \ldots, d_{r^{i,j}}^{i,j} \ge 1\\
    n_{r^{i,j}+1}^{i,j}, \ldots, n_{r^{i,j}+s^{i,j}}^{i,j} \ge 1, 
    d_{r^{i,j}+1}^{i,j}, \ldots, d_{r^{i,j}+s^{i,j}}^{i,j} \ge 1}}
  \prod_{k=1}^{r^{i,j}+s^{i,j}} (n_k^{i,j})^{u_k^{i,j}}
    (d_k^{i,j})^{t_k^{i,j}} 
\cdot q^{\sum_{k=1}^{r^{i,j}+s^{i,j}} n_k^{i,j} d_k^{i,j}}
$$
\begin{eqnarray}    \label{finite_redu1toNm>0.0}
\cdot \prod_{1 \le i \le N}
\sum_{\substack{n_1^{(i)}, \ldots, n_{r_i}^{(i)} \ge 0, 
    d_1^{(i)}, \ldots, d_{r_i}^{(i)} \ge 1\\
    n_{r_i+1}^{(i)}, \ldots, n_{r_i+s_i}^{(i)} \ge 1,
    d_{r_i+1}^{(i)}, \ldots, d_{r_i+s_i}^{(i)} \ge 1}}
  \prod_{k=1}^{r_i+s_i} (n_k^{(i)})^{u_k^{(i)}} (d_k^{(i)})^{t_k^{(i)}}
\cdot q^{\sum_{k=1}^{r_i+s_i} n_k^{(i)} d_k^{(i)}}
\end{eqnarray}    
together with the condition that for every $1 \le i \le N$,
\begin{eqnarray*}
& &\sum_{1 \le j < i} \bigg (\sum_{k=1}^{r^{j,i}} d_k^{j,i} 
    - \sum_{k=1}^{s^{j,i}} d_{r^{j,i}+k}^{j,i} \bigg )
  + \bigg (\sum_{k=1}^{r_i} d_k^{(i)} 
    - \sum_{k=1}^{s_i} d_{r_i+k}^{(i)} \bigg )    \\
&=&\sum_{N \ge j > i} \bigg (\sum_{k=1}^{r^{i,j}} d_k^{i,j} 
    - \sum_{k=1}^{s^{i,j}} d_{r^{i,j}+k}^{i,j} \bigg ),
\end{eqnarray*}
is a finite $\Q$-linear combination of expressions of the form
$$
\prod_{1 \le i < j \le N}
\sum_{\substack{\w n_1^{i,j} > \ldots > \w n_{f^{i,j}}^{i,j} \ge 0, 
    \w d_1^{i,j}, \ldots, \w d_{f^{i,j}}^{i,j} \ge 1\\
    \w n_{f^{i,j}+1}^{i,j} > \ldots > \w n_{f^{i,j}+g^{i,j}}^{i,j} \ge 1, 
    \w d_{f^{i,j}+1}^{i,j}, \ldots, \w d_{f^{i,j}+g^{i,j}}^{i,j} \ge 1}}
\prod_{k=1}^{f^{i,j}+g^{i,j}} (\w n_k^{i,j})^{\w u_k^{i,j}}
    (\w d_k^{i,j})^{\w t_k^{i,j}} 
\cdot q^{\sum_{k=1}^{f^{i,j}+g^{i,j}} \w n_k^{i,j} \w d_k^{i,j}}
$$
\begin{eqnarray}    \label{finite_redu1toNm>0.01}
\cdot \prod_{1 \le i \le N}
\sum_{\substack{\w n_1^{(i)} > \ldots > \w n_{f_i}^{(i)} \ge 0, 
    \w d_1^{(i)}, \ldots, \w d_{f_i}^{(i)} \ge 1\\
    \w n_{f_i+1}^{(i)} > \ldots > \w n_{f_i+g_i}^{(i)} \ge 1,
    \w d_{f_i+1}^{(i)}, \ldots, \w d_{f_i+g_i}^{(i)} \ge 1}}
  \prod_{k=1}^{f_i+g_i} (\w n_k^{(i)})^{\w u_k^{(i)}} 
    (\w d_k^{(i)})^{\w t_k^{(i)}}
\cdot q^{\sum_{k=1}^{f_i+g_i} \w n_k^{(i)} \w d_k^{(i)}}
\end{eqnarray}
together with the condition that for every $1 \le i \le N$,
\begin{eqnarray*}
& &\sum_{1 \le j < i} \bigg (\sum_{k=1}^{f^{j,i}} \w d_k^{j,i} 
    - \sum_{k=1}^{g^{j,i}} \w d_{f^{j,i}+k}^{j,i} \bigg )   
   + \bigg (\sum_{k=1}^{f_i} \w d_k^{(i)} 
    - \sum_{k=1}^{g_i} \w d_{f_i+k}^{(i)} \bigg )   \\
&=&\sum_{N \ge j > i} \bigg (\sum_{k=1}^{f^{i,j}} \w d_k^{i,j} 
    - \sum_{k=1}^{g^{i,j}} \w d_{f^{i,j}+k}^{i,j} \bigg )
\end{eqnarray*}
where $\displaystyle{\sum_{1 \le i < j \le N} \sum_{k=1}^{f^{i,j}+g^{i,j}} 
(\w u_k^{i,j} + \w t_k^{i,j} + 1) 
+ \sum_{1 \le i \le N} \sum_{k=1}^{f_i+g_i} 
(\w u_k^{(i)} + \w t_k^{(i)} +1)}$ is at most
$$
\sum_{1 \le i < j \le N} \sum_{k=1}^{r^{i,j}+s^{i,j}} 
(u_k^{i,j} + t_k^{i,j} + 1) 
+ \sum_{1 \le i \le N} \sum_{k=1}^{r_i+s_i} (u_k^{(i)} +t_k^{(i)} +1).
$$ 
Moreover, if $t_k^{i,j}, t_k^{(i')} \ge 1$ for all $k$, $i < j$ and $i'$, 
then $\w t_k^{i,j}, \w t_k^{(i')} \ge 1$ for all $k$, $i < j$ and $i'$. 
If $u_k^{i,j}, u_k^{(i')} \ge 1$ for all $k$, $i < j$ and $i'$, 
then $\w u_k^{i,j}, \w u_k^{(i')} \ge 1$ for all $k$, $i < j$ and $i'$.
\end{lemma}
\begin{proof}
Follows by applying the proof of Lemma~\ref{finite_redu} to 
each factor in the products $\prod_{1 \le i < j \le N}$ and 
$\prod_{1 \le i \le N}$ in \eqref{finite_redu1toNm>0.0}.
\end{proof}

Next, we prove that the expression \eqref{finite_redu1toNm>0.01} 
is contained in $\BD$.

\begin{lemma}  \label{lemma_t2t1yq1toNm>0} 
Let $N \ge 2$. For $1 \le i < j \le N$, fix $r^{i,j}, s^{i,j} \ge 0$. 
For $1 \le i < j \le N$ and $1 \le k \le r^{i,j} + s^{i,j}$, 
fix $u_k^{i,j}, t_k^{i,j} \ge 0$. 
For $1 \le i \le N$, fix $r_i, s_i \ge 0$. For $1 \le i \le N$
and $1 \le k \le r_i + s_i$, fix $u_k^{(i)}, t_k^{(i)} \ge 0$. 
Let $V^{(N)}$ denote
$$
\prod_{1 \le i < j \le N}
\sum_{\substack{n_1^{i,j} > \ldots > n_{r^{i,j}}^{i,j} \ge 0, 
    d_1^{i,j}, \ldots, d_{r^{i,j}}^{i,j} \ge 1\\
    n_{r^{i,j}+1}^{i,j} > \ldots > n_{r^{i,j}+s^{i,j}}^{i,j} \ge 1, 
    d_{r^{i,j}+1}^{i,j}, \ldots, d_{r^{i,j}+s^{i,j}}^{i,j} \ge 1}}
  \prod_{k=1}^{r^{i,j}+s^{i,j}} (n_k^{i,j})^{u_k^{i,j}}
    (d_k^{i,j})^{t_k^{i,j}} 
\cdot q^{\sum_{k=1}^{r^{i,j}+s^{i,j}} n_k^{i,j} d_k^{i,j}}
$$
\begin{eqnarray}    \label{lemma_t2t1yq1toNm>0.01}
\cdot \prod_{1 \le i \le N}
\sum_{\substack{n_1^{(i)} > \ldots > n_{r_i}^{(i)} \ge 0, 
    d_1^{(i)}, \ldots, d_{r_i}^{(i)} \ge 1\\
    n_{r_i+1}^{(i)} > \ldots > n_{r_i+s_i}^{(i)} \ge 1,
    d_{r_i+1}^{(i)}, \ldots, d_{r_i+s_i}^{(i)} \ge 1}}
  \prod_{k=1}^{r_i+s_i} (n_k^{(i)})^{u_k^{(i)}} (d_k^{(i)})^{t_k^{(i)}}
\cdot q^{\sum_{k=1}^{r_i+s_i} n_k^{(i)} d_k^{(i)}}
\end{eqnarray}    
together with the condition that for every $1 \le i \le N$,
\begin{eqnarray}    \label{lemma_t2t1yq1toNm>0.02}
& &\sum_{1 \le j < i} \bigg (\sum_{k=1}^{r^{j,i}} d_k^{j,i} 
    - \sum_{k=1}^{s^{j,i}} d_{r^{j,i}+k}^{j,i} \bigg )
  + \bigg (\sum_{k=1}^{r_i} d_k^{(i)} 
    - \sum_{k=1}^{s_i} d_{r_i+k}^{(i)} \bigg )  \nonumber  \\
&=&\sum_{N \ge j > i} \bigg (\sum_{k=1}^{r^{i,j}} d_k^{i,j} 
    - \sum_{k=1}^{s^{i,j}} d_{r^{i,j}+k}^{i,j} \bigg ).
\end{eqnarray}
Then, $V^{(N)} \in \BD$ with weight at most   
$$
\sum_{1 \le i < j \le N} \sum_{k=1}^{r^{i,j}+s^{i,j}} 
  (u_k^{i,j} + t_k^{i,j} + 1)
+ \sum_{1 \le i \le N} \sum_{k=1}^{r_i+s_i} (u_k^{(i)} +t_k^{(i)} +1).
$$
Moreover, if $t_k^{i,j}, t_k^{(i')} \ge 1$ for all $k$, $i < j$ and $i'$, 
then $V^{(N)} \in \qBD$. 
\end{lemma}
\begin{proof}
Note that for $1 \le i \le N$ and $1 \le k \le r_i+s_i$, 
the symbol $d_{k}^{(i)}$ appears exactly once in the $N$ equations 
\eqref{lemma_t2t1yq1toNm>0.02} (that is, 
appears on the {\it left-hand-side} of the $i$-th equation). 
Applying Remark~\ref{rmk_rs} to each of the $N$ factors in 
$\displaystyle{\prod_{1 \le i \le N}}$ from 
\eqref{lemma_t2t1yq1toNm>0.01} successively, 
we are able to eliminate the symbols $d_{k}^{(i)}$ 
satisfying $1 \le i \le N$ and $1 \le k \le r_i$ 
from the the $N$ equations \eqref{lemma_t2t1yq1toNm>0.02}.
Therefore, we conclude that $V^{(N)}$ is 
a finite $\Q$-linear combination of expressions of the form 
$$
A_1 \cdot \prod_{1 \le i < j \le N}
\sum_{\substack{n_1^{i,j} > \ldots > n_{r^{i,j}}^{i,j} \ge 0, 
    d_1^{i,j}, \ldots, d_{r^{i,j}}^{i,j} \ge 1\\
    n_{r^{i,j}+1}^{i,j} > \ldots > n_{r^{i,j}+s^{i,j}}^{i,j} \ge 1\\
    d_{r^{i,j}+1}^{i,j}, \ldots, d_{r^{i,j}+s^{i,j}}^{i,j} \ge 1}}
  \prod_{k=1}^{r^{i,j}+s^{i,j}} (n_k^{i,j})^{u_k^{i,j}}
    (d_k^{i,j})^{\w t_k^{i,j}} 
\cdot q^{\sum_{k=1}^{r^{i,j}+s^{i,j}} n_k^{i,j} d_k^{i,j}}
$$
\begin{eqnarray}    \label{lemma_t2t1yq1toNm>0.1}
\cdot \prod_{i \in I_1} \,\,
\sum_{\substack{n_{i,1} > \ldots > n_{i, r_i} \ge 1, 
    d_{i,1}, \ldots, d_{i, r_i} \ge 1}}
\prod_{k=1}^{r_i} n_{i,k}^{u_{i,k}} d_{i,k}^{t_{i,k}} 
\cdot q^{\sum_{k=1}^{r_i} n_{i,k} d_{i,k}}
\end{eqnarray}
for some subset $I_1 \subset \{1, \ldots, N\}$,
together with the $|I_1|$ equations
\begin{eqnarray}    \label{lemma_t2t1yq1toNm>0.2}
\sum_{1 \le j < i} \bigg (\sum_{k=1}^{r^{j,i}} d_k^{j,i} 
    - \sum_{k=1}^{s^{j,i}} d_{r^{j,i}+k}^{j,i} \bigg )
= \sum_{N \ge j > i} \bigg (\sum_{k=1}^{r^{i,j}} d_k^{i,j} 
    - \sum_{k=1}^{s^{i,j}} d_{r^{i,j}+k}^{i,j} \bigg )
  + \sum_{k=1}^{r_i} d_{i,k}
\end{eqnarray}
indexed by $i \in I_1$ (these equations come from the void equation 
\eqref{rmk_rs.100} modified in our present setting). 
In the above, $A_1 \in \BD$ and 
\begin{eqnarray*}
& &{\rm wt}(A_1) 
   + \sum_{1 \le i < j \le N-1} \sum_{k=1}^{r^{i,j}+s^{i,j}} 
   (u_k^{i,j} + \w t_k^{i,j} + 1)
   + \sum_{i \in I_1} \sum_{k=1}^{r_i} (u_{i,k} + t_{i,k} + 1)  \\
&\le&\sum_{1 \le i < j \le N} \sum_{k=1}^{r^{i,j}+s^{i,j}} 
   (u_k^{i,j} + t_k^{i,j} + 1).
\end{eqnarray*}
Moreover, if $t_k^{i,j}, t_k^{(i')} \ge 1$ for all $k$, 
$1 \le i < j \le N$ and $1 \le i' \le N$,
then $A_1 \in \qBD$ and $\w t_k^{i,j}, t_{i',1} \ge 1$ 
for all $k$, $1 \le i < j \le N$ and $i' \in I_1$. 

Put $I_0 = \{1, \ldots, N\}$. We rewrite \eqref{lemma_t2t1yq1toNm>0.1} as
$$
A_1 \cdot \prod_{\substack{i < j\\i, j \in I_0}}
\sum_{\substack{n_1^{i,j} > \ldots > n_{r^{i,j}}^{i,j} \ge 0, 
    d_1^{i,j}, \ldots, d_{r^{i,j}}^{i,j} \ge 1\\
    n_{r^{i,j}+1}^{i,j} > \ldots > n_{r^{i,j}+s^{i,j}}^{i,j} \ge 1\\
    d_{r^{i,j}+1}^{i,j}, \ldots, d_{r^{i,j}+s^{i,j}}^{i,j} \ge 1}}
  \prod_{k=1}^{r^{i,j}+s^{i,j}} (n_k^{i,j})^{u_k^{i,j}}
    (d_k^{i,j})^{\w t_k^{i,j}} 
\cdot q^{\sum_{k=1}^{r^{i,j}+s^{i,j}} n_k^{i,j} d_k^{i,j}}
$$
\begin{eqnarray}    \label{lemma_t2t1yq1toNm>0.100}
\cdot \prod_{i \in I_1} \,\,
\sum_{\substack{n_{i,1} > \ldots > n_{i, r_i} \ge 1, 
    d_{i,1}, \ldots, d_{i, r_i} \ge 1}}
\prod_{k=1}^{r_i} n_{i,k}^{u_{i,k}} d_{i,k}^{t_{i,k}} 
\cdot q^{\sum_{k=1}^{r_i} n_{i,k} d_{i,k}}.
\end{eqnarray}
Similarly, we rewrite \eqref{lemma_t2t1yq1toNm>0.2} with $i \in I_1$ as
\begin{eqnarray}    \label{lemma_t2t1yq1toNm>0.101}
\sum_{\substack{j < i\\i, j \in I_0}} \bigg (\sum_{k=1}^{r^{j,i}} d_k^{j,i} 
    - \sum_{k=1}^{s^{j,i}} d_{r^{j,i}+k}^{j,i} \bigg )
= \sum_{\substack{j < i\\i, j \in I_0}} 
    \bigg (\sum_{k=1}^{r^{i,j}} d_k^{i,j} 
    - \sum_{k=1}^{s^{i,j}} d_{r^{i,j}+k}^{i,j} \bigg )
  + \sum_{k=1}^{r_i} d_{i,k}
\end{eqnarray}

If $|I_1| = |I_0|$, then $I_1 = I_0$. 
Adding up the $|I_1|$ equations \eqref{lemma_t2t1yq1toNm>0.101} yields
$$
\sum_{i \in I_1} \sum_{k=1}^{r_i} d_{i,k} = 0.
$$
So $r_i = 0$ for all $i \in I_1$. In addition,
since it is redundant, we can drop any equation from 
the $|I_1|$ equations \eqref{lemma_t2t1yq1toNm>0.101}. 
Thus the expression \eqref{lemma_t2t1yq1toNm>0.1} becomes 
$$
A_1 \cdot \prod_{\substack{i < j\\i, j \in I_0}}
\sum_{\substack{n_1^{i,j}, \ldots, n_{r^{i,j}}^{i,j} \ge 0, 
    d_1^{i,j}, \ldots, d_{r^{i,j}}^{i,j} \ge 1\\
    n_{r^{i,j}+1}^{i,j}, \ldots, n_{r^{i,j}+s^{i,j}}^{i,j} \ge 1\\
    d_{r^{i,j}+1}^{i,j}, \ldots, d_{r^{i,j}+s^{i,j}}^{i,j} \ge 1}}
  \prod_{k=1}^{r^{i,j}+s^{i,j}} (n_k^{i,j})^{u_k^{i,j}}
    (d_k^{i,j})^{\w t_k^{i,j}} 
\cdot q^{\sum_{k=1}^{r^{i,j}+s^{i,j}} n_k^{i,j} d_k^{i,j}}
$$
together with the condition that for $i \in I_0 - \{N\}$,
$$
  \sum_{1 \le j < i} \bigg (\sum_{k=1}^{r^{j,i}} d_k^{j,i} 
    - \sum_{k=1}^{s^{j,i}} d_{r^{j,i}+k}^{j,i} \bigg )
= \sum_{N-1 \ge j > i} \bigg (\sum_{k=1}^{r^{i,j}} d_k^{i,j} 
    - \sum_{k=1}^{s^{i,j}} d_{r^{i,j}+k}^{i,j} \bigg ).
$$
In other words, this case is the same as the case when 
$|I_1| = |I_0|-1$ and $I_1 = I_0 - \{N\}$.
Therefore, we assume $|I_1| \le |I_0|-1$ below. 

Put $J_1 = \{1, \ldots, N\} - I_1$.
Since $|I_1| \le N-1$, $J_1 \ne \emptyset$. 
Let $i_1 \in I_1$ and $j_1 \in J_1$.
If $j_1 < i_1$, then for each $1 \le k \le r^{j_1, i_1}+s^{j_1, i_1}$,
the symbol $d_{k}^{j_1,i_1}$ appears exactly once in 
the $|I_1|$ equations \eqref{lemma_t2t1yq1toNm>0.2} (it appears on 
the {\it left-hand-side} of the equation indexed by $i_1$). 
Again, we apply Remark~\ref{rmk_rs} to 
the factors of $\displaystyle{\prod_{1 \le i < j \le N}}$ in 
\eqref{lemma_t2t1yq1toNm>0.1} indexed by all these ordered pairs $j_1<i_1$ 
to eliminate all the symbols $d_{k}^{j_1, i_1}$ satisfying 
$1 \le k \le r^{j_1, i_1}$ from the $|I_1|$ equations 
\eqref{lemma_t2t1yq1toNm>0.2}. On the other hand, if $i_1 < j_1$, 
then for each $1 \le k \le r^{i_1,j_1}+s^{i_1,j_1}$,
the symbol $d_{k}^{i_1,j_1}$ appears exactly once in the $|I_1|$ equations \eqref{lemma_t2t1yq1toNm>0.2} (it appears on 
the {\it right-hand-side} of the equation indexed by $i_1$). 
This time, we apply the proof of Lemma~\ref{lemma_t2t1yq} to 
the factors of $\displaystyle{\prod_{1 \le i < j \le N}}$ 
in \eqref{lemma_t2t1yq1toNm>0.1} indexed by all these ordered pairs 
$i_1<j_1$ to eliminate all the symbols $d_{k}^{i_1,j_1}$ satisfying 
$r^{i_1,j_1}+ 1 \le k \le r^{i_1,j_1}+s^{i_1,j_1}$ 
from the $|I_1|$ equations \eqref{lemma_t2t1yq1toNm>0.2}. 
So $V^{(N)}$ is 
a finite $\Q$-linear combination of expressions of the form 
$$
A_2 \cdot \prod_{\substack{i < j\\i, j \in I_1}}
\sum_{\substack{n_1^{i,j} > \ldots > n_{r^{i,j}}^{i,j} \ge 0, 
    d_1^{i,j}, \ldots, d_{r^{i,j}}^{i,j} \ge 1\\
    n_{r^{i,j}+1}^{i,j} > \ldots > n_{r^{i,j}+s^{i,j}}^{i,j} \ge 1\\
    d_{r^{i,j}+1}^{i,j}, \ldots, d_{r^{i,j}+s^{i,j}}^{i,j} \ge 1}}
  \prod_{k=1}^{r^{i,j}+s^{i,j}} (n_k^{i,j})^{u_k^{i,j}}
    (d_k^{i,j})^{\hat t_k^{i,j}} 
\cdot q^{\sum_{k=1}^{r^{i,j}+s^{i,j}} n_k^{i,j} d_k^{i,j}}
$$
\begin{eqnarray}    \label{lemma_t2t1yq1toNm>0.3}
\cdot \prod_{i \in I_2} \,\,
\sum_{\substack{n_{i,1} > \ldots > n_{i, r_i} \ge 1, 
    \w d_{i,1}, \ldots, \w d_{i, \w r_i} \ge 1}}
\prod_{k=1}^{\w r_i} n_{i,k}^{\w u_{i,k}} \w d_{i,k}^{\w t_{i,k}} 
\cdot q^{\sum_{k=1}^{r_i} n_{i,k} \w d_{i,k}}
\end{eqnarray}
for some subset $I_2 \subset I_1$, together with the $|I_2|$ equations
\begin{eqnarray}    \label{lemma_t2t1yq1toNm>0.4}
  \sum_{\substack{j < i\\i, j \in I_1}} 
  \bigg (\sum_{k=1}^{r^{j,i}} d_k^{j,i} 
    - \sum_{k=1}^{s^{j,i}} d_{r^{j,i}+k}^{j,i} \bigg )
= \sum_{\substack{i < j\\i, j \in I_1}} 
  \bigg (\sum_{k=1}^{r^{i,j}} d_k^{i,j} 
    - \sum_{k=1}^{s^{i,j}} d_{r^{i,j}+k}^{i,j} \bigg )
  + \sum_{k=1}^{\w r_i} \w d_{i,k}
\end{eqnarray}
indexed by $i \in I_2$. In the above expression, $A_2 \in \BD$ and 
\begin{eqnarray*}
& &{\rm wt}(A_2) 
   + \sum_{\substack{i < j\\i, j \in I_1}} \sum_{k=1}^{r^{i,j}+s^{i,j}} 
   (u_k^{i,j} + \hat t_k^{i,j} + 1)
   + \sum_{i \in I_2} \sum_{k=1}^{\w r_i} (\w u_{i,k} + \w t_{i,k} + 1)  \\
&\le&\sum_{1 \le i < j \le N} \sum_{k=1}^{r^{i,j}+s^{i,j}} 
   (u_k^{i,j} + t_k^{i,j} + 1).
\end{eqnarray*}
Moreover, if $t_k^{i,j}, t_k^{(i')} \ge 1$ for all $k$, 
$1 \le i < j \le N$ and $1 \le i' \le N$,
then $A_2 \in \qBD$ and $\w t_k^{i,j}, t_{i',1} \ge 1$ 
for all $k$, $i < j$ with $i,j \in I_1$, and $i' \in I_2$. 
As in the preceding paragraph, we may assume $|I_2| < |I_1|$. 
Put $J_2 = I_1 - I_2$. Then, $J_2 \ne \emptyset$.
Now repeating the above arguments, we complete the proof of the lemma.
\end{proof}

\begin{proposition}  \label{Prop_t2t1yq1toNm>0} 
Let $N \ge 2$. For $1 \le i < j \le N$, fix $r^{i,j}, s^{i,j} \ge 0$. 
For $1 \le i < j \le N$ and $1 \le k \le r^{i,j} + s^{i,j}$, 
fix $u_k^{i,j}, t_k^{i,j} \ge 0$. 
For $1 \le i \le N$, fix $r_i, s_i \ge 0$. For $1 \le i \le N$
and $1 \le k \le r_i + s_i$, fix $u_k^{(i)}, t_k^{(i)} \ge 0$. 
Let $\W V^{(N)}$ denote
$$
\prod_{1 \le i < j \le N}
\sum_{\substack{n_1^{i,j}, \ldots, n_{r^{i,j}}^{i,j} \ge 0, 
    d_1^{i,j}, \ldots, d_{r^{i,j}}^{i,j} \ge 1\\
    n_{r^{i,j}+1}^{i,j}, \ldots, n_{r^{i,j}+s^{i,j}}^{i,j} \ge 1, 
    d_{r^{i,j}+1}^{i,j}, \ldots, d_{r^{i,j}+s^{i,j}}^{i,j} \ge 1}}
  \prod_{k=1}^{r^{i,j}+s^{i,j}} (n_k^{i,j})^{u_k^{i,j}}
    (d_k^{i,j})^{t_k^{i,j}} 
\cdot q^{\sum_{k=1}^{r^{i,j}+s^{i,j}} n_k^{i,j} d_k^{i,j}}
$$
$$
\cdot \prod_{1 \le i \le N}
\sum_{\substack{n_1^{(i)}, \ldots, n_{r_i}^{(i)} \ge 0, 
    d_1^{(i)}, \ldots, d_{r_i}^{(i)} \ge 1\\
    n_{r_i+1}^{(i)}, \ldots, n_{r_i+s_i}^{(i)} \ge 1,
    d_{r_i+1}^{(i)}, \ldots, d_{r_i+s_i}^{(i)} \ge 1}}
  \prod_{k=1}^{r_i+s_i} (n_k^{(i)})^{u_k^{(i)}} (d_k^{(i)})^{t_k^{(i)}}
\cdot q^{\sum_{k=1}^{r_i+s_i} n_k^{(i)} d_k^{(i)}}
$$
together with the condition that for every $1 \le i \le N$,
\begin{eqnarray*}
& &\sum_{1 \le j < i} \bigg (\sum_{k=1}^{r^{j,i}} d_k^{j,i} 
    - \sum_{k=1}^{s^{j,i}} d_{r^{j,i}+k}^{j,i} \bigg )
  + \bigg (\sum_{k=1}^{r_i} d_k^{(i)} 
    - \sum_{k=1}^{s_i} d_{r_i+k}^{(i)} \bigg )    \\
&=&\sum_{N \ge j > i} \bigg (\sum_{k=1}^{r^{i,j}} d_k^{i,j} 
    - \sum_{k=1}^{s^{i,j}} d_{r^{i,j}+k}^{i,j} \bigg ),
\end{eqnarray*}
Then, $\W V^{(N)} \in \BD$ with weight at most   
$$
\sum_{1 \le i < j \le N} \sum_{k=1}^{r^{i,j}+s^{i,j}} 
  (u_k^{i,j} + t_k^{i,j} + 1)
+ \sum_{1 \le i \le N} \sum_{k=1}^{r_i+s_i} (u_k^{(i)} +t_k^{(i)} +1).
$$
Moreover, if $t_k^{i,j}, t_k^{(i')} \ge 1$ for all $k$, $i < j$ and $i'$, 
then $\W V^{(N)} \in \qBD$. 
\end{proposition}
\begin{proof}
Follows immediately from Lemma~\ref{finite_redu1toNm>0} and 
Lemma~\ref{lemma_t2t1yq1toNm>0}.
\end{proof}

The following result is our Theorem~\ref{Intro-yij1toNm>0}.

\begin{theorem}  \label{yij1toNm>0} 
Let notations be from Definition~\ref{def-zmvn}. Then,
\begin{enumerate}
\item[{\rm (i)}]
for $m_\ell^{i,j}, n_\ell^{i,j}, m_i, n_i \ge 0$, the coefficient of 
$$
\prod_{1 \le i < j \le N, 1 \le \ell \le 2} (z_\ell^{i,j})^{m_\ell^{i,j}}
(v_\ell^{i,j})^{n_\ell^{i,j}} \cdot \prod_{1 \le i \le N} 
z_i^{m_i} v_i^{n_i}
$$ 
in 
$
\Coe_{y_1^0 \cdots y_N^0}(\mathfrak P_{N}^{a,b})
$ 
is contained in $\BD[a, b]$ with weight at most 
\begin{eqnarray}    \label{yij1toNm>0.01}
\sum_{1 \le i < j \le N, 1 \le \ell \le 2} (m_\ell^{i,j} + n_\ell^{i,j})
+ \sum_{1 \le i \le N} (m_i + n_i)
\end{eqnarray}
and with degree in ($a$ and $b$) at most 
$
    \sum_{1 \le i \le N} (m_i + n_i)
$.

\item[{\rm (ii)}] 
for $a=b=0$ and $m_\ell^{i,j}, n_\ell^{i,j}\ge 0$, the coefficient of 
$$
\prod_{1 \le i < j \le N, 1 \le \ell \le 2} (z_\ell^{i,j})^{m_\ell^{i,j}}
(v_\ell^{i,j})^{n_\ell^{i,j}}
$$ 
in 
$
\Coe_{y_1^0 \cdots y_N^0}(\mathfrak P_{N}^{0,0})
$ 
is contained in $\qBD$ with weight at most 
\begin{eqnarray}    \label{yij1toNm>0.03}
\sum_{1 \le i < j \le N, 1 \le \ell \le 2} (m_\ell^{i,j} + n_\ell^{i,j}).
\end{eqnarray}
\end{enumerate}
\end{theorem}
\begin{proof}
(i) By \eqref{t2t1yq.50}, the coefficient of 
$\prod_{1 \le \ell \le 2} (z_\ell^{i,j})^{m_\ell^{i,j}}
(v_\ell^{i,j})^{n_\ell^{i,j}}$ in 
$$
\frac{(x_1^{i,j}x_2^{i,j}y_{i,j})_\infty (y_{i,j})_\infty}
     {(x_1^{i,j}y_{i,j})_\infty (x_2^{i,j}y_{i,j})_\infty} 
   \frac{(u_1^{i,j}u_2^{i,j}y_{i,j}^{-1}q)_\infty (y_{i,j}^{-1}q)_\infty}
     {(u_1^{i,j}y_{i,j}^{-1}q)_\infty (u_2^{i,j}y_{i,j}^{-1}q)_\infty}
$$
is a finite $\Q$-linear combination of expressions of the form
$$ 
\sum_{\substack{n_1, \ldots, n_r \ge 0, d_1, \ldots, d_r \ge 1\\
    n_{r+1}, \ldots, n_{r+s} \ge 1, d_{r+1}, \ldots, d_{r+s} \ge 1}}
  \prod_{k=1}^{r+s} d_k^{t_k} 
  \cdot q^{\sum_{k=1}^{r+s} n_k d_k}
  y_{i,j}^{\sum_{k=1}^r d_k - \sum_{k=1}^{s} d_{r+k}}
$$
where $t_k \ge 1$ for every $k$ and $\sum_{k=1}^{r+s} (t_{k} +1) 
= \sum_{1 \le \ell \le 2} (m_\ell^{i,j} + n_\ell^{i,j})$. 
So the coefficient of $\prod_{1 \le i < j \le N, 1 \le \ell \le 2} 
(z_\ell^{i,j})^{m_\ell^{i,j}}(v_\ell^{i,j})^{n_\ell^{i,j}}$ in 
\begin{eqnarray}    \label{yij1toNm>0.99}
\prod_{1 \le i < j \le N} 
   \frac{(x_1^{i,j}x_2^{i,j}y_{i,j})_\infty (y_{i,j})_\infty}
     {(x_1^{i,j}y_{i,j})_\infty (x_2^{i,j}y_{i,j})_\infty} 
   \frac{(u_1^{i,j}u_2^{i,j}y_{i,j}^{-1}q)_\infty (y_{i,j}^{-1}q)_\infty}
     {(u_1^{i,j}y_{i,j}^{-1}q)_\infty (u_2^{i,j}y_{i,j}^{-1}q)_\infty}
\end{eqnarray}
is a finite $\Q$-linear combination of expressions of the form
$$ 
\prod_{1 \le i < j \le N}
\sum_{\substack{n_1^{i,j}, \ldots, n_{r^{i,j}}^{i,j} \ge 0, 
    d_1^{i,j}, \ldots, d_{r^{i,j}}^{i,j} \ge 1\\
    n_{r^{i,j}+1}^{i,j}, \ldots, n_{r^{i,j}+s^{i,j}}^{i,j} \ge 1, 
    d_{r^{i,j}+1}^{i,j}, \ldots, d_{r^{i,j}+s^{i,j}}^{i,j} \ge 1}}
  \prod_{k=1}^{r^{i,j}+s^{i,j}} (d_k^{i,j})^{t_k^{i,j}} 
$$
\begin{eqnarray}    \label{yij1toNm>0.100}
\cdot q^{\sum_{k=1}^{r^{i,j}+s^{i,j}} n_k^{i,j} d_k^{i,j}}
  y_{i,j}^{\sum_{k=1}^{r^{i,j}} d_k^{i,j} 
    - \sum_{k=1}^{s^{i,j}} d_{r+k}^{i,j}}
\end{eqnarray}
where $t_k^{i,j} \ge 1$ for all $k$ and $1 \le i < j \le N$, and
\begin{eqnarray}    \label{yij1toNm>0.101}
  \sum_{1 \le i < j \le N} \sum_{k=1}^{r^{i,j}+s^{i,j}} (t_k^{i,j} +1)
= \sum_{1 \le i < j \le N, 1 \le \ell \le 2} 
   (m_\ell^{i,j} + n_\ell^{i,j}).
\end{eqnarray}

Next, regard $((x_iy_i)_\infty (y_i)_\infty^{-1})^a$ 
as an element in $\Q[[z_i,q,y_i,a]]$. Put
$$
f(z_i) = a \ln\big ((x_iy_i)_\infty (y_i)_\infty^{-1} \big ).
$$
Then, $f(0) = 0$. For $m \ge 0$, let $a_{m} = \Coe_{z_i^m} 
((x_iy_i)_\infty (y_i)_\infty^{-1})^a \in \C[[q,y_i, a]]$. 
Then,
\begin{eqnarray}   \label{yij1toNm>0.1}
  a_{m}
= \sum_{k_1+2k_2+3k_3+\cdots = m}
  \frac{1}{k_1! k_2! k_3! \cdots} \prod_{t \ge 1}
  \left (\frac{D_{x_i}^t f}{t!} \Big |_{x_i=1}\right )^{k_t}
\end{eqnarray} 
by \eqref{AAA06250955am.1}.
Since $f(z_i) = a \ln\big ((y_i)_\infty^{-1} \big ) 
+ a \sum_{n \ge 0} \ln (1-x_iy_iq^n)$, we have
$$ 
   D_{x_i}f  
= -a \sum_{n \ge 0} \frac{x_iy_iq^n}{1-x_iy_iq^n}  
= -a \sum_{n \ge 0} \sum_{d \ge 1} (x_iy_iq^n)^d.
$$
So for $t \ge 1$, we get 
$$
  D_{x_i}^tf|_{x_i=1} 
= -a \sum_{n \ge 0} \sum_{d \ge 1} d^{t-1} (y_iq^n)^d. 
$$
By \eqref{yij1toNm>0.1}, for $m \ge 1$, $a_m$ is 
a finite $\Q$-linear combination of expressions of the form:
\begin{eqnarray}   \label{yij1toNm>0.2}
a^r \prod_{j=1}^r \left (\sum_{n_j \ge 0} \sum_{d_j \ge 1} 
d_j^{t_j-1} y^{d_j}q^{n_j d_j} \right )
\end{eqnarray}
where $t_j \ge 1$ and $r \le \sum_{j=1}^r t_j = m$. 
Note that since $a_0 = 1$, 
we have
\begin{eqnarray}   \label{yij1toNm>0.3}
   \bigg (\frac{(x_iy_i)_\infty}{(y_i)_\infty} \bigg )^a
= 1 + \sum_{m \ge 1} a_m z_i^m
\in 1 + \Q[[q, y_i, a, z_i]] \cdot z_i. 
\end{eqnarray}

Similarly, for $n \ge 1$, the coefficient of $v_i^n$ in 
$$
   \bigg (\frac{(u_iy_i^{-1}q)_\infty}{(y_i^{-1}q)_\infty} \bigg )^b
\in 1 + \Q[[q, y_i^{-1}, b, v_i]] \cdot v_i 
$$
is a finite $\Q$-linear combination of expressions of the form:
\begin{eqnarray}   \label{yij1toNm>0.4}
b^s \prod_{j=1}^s \left (\sum_{n_j \ge 0} \sum_{d_j \ge 1} 
d_j^{t_j-1} (y^{-1})^{d_j}q^{n_j d_j} \right )
\end{eqnarray}
where $t_j \ge 1$ and $s \le \sum_{j=1}^s t_j = n$. For simplicity, put 
$$
  {\rm z^mv^n} 
= \prod_{1 \le i < j \le N, 1 \le \ell \le 2} (z_\ell^{i,j})^{m_\ell^{i,j}}
  (v_\ell^{i,j})^{n_\ell^{i,j}} \cdot \prod_{1 \le i \le N} 
  z_i^{m_i} v_i^{n_i}.
$$
By \eqref{yij1toNm>0.100}, \eqref{yij1toNm>0.2} and \eqref{yij1toNm>0.4},
$\Coe_{{\rm z^mv^n}} \mathfrak P_N^{a,b}$
is a finite $\Q$-linear combination of expressions of the form
$$ 
\prod_{1 \le i < j \le N}
\sum_{\substack{n_1^{i,j}, \ldots, n_{r^{i,j}}^{i,j} \ge 0, 
    d_1^{i,j}, \ldots, d_{r^{i,j}}^{i,j} \ge 1\\
    n_{r^{i,j}+1}^{i,j}, \ldots, n_{r^{i,j}+s^{i,j}}^{i,j} \ge 1\\
    d_{r^{i,j}+1}^{i,j}, \ldots, d_{r^{i,j}+s^{i,j}}^{i,j} \ge 1}}
  \prod_{k=1}^{r^{i,j}+s^{i,j}} (d_k^{i,j})^{t_k^{i,j}} 
\cdot q^{\sum_{k=1}^{r^{i,j}+s^{i,j}} n_k^{i,j} d_k^{i,j}}
  y_{i,j}^{\sum_{k=1}^{r^{i,j}} d_k^{i,j} 
    - \sum_{k=1}^{s^{i,j}} d_{r^{i,j}+k}^{i,j}}
$$
$$
\cdot a^rb^s \prod_{1 \le i \le N}
\sum_{\substack{n_1^{(i)}, \ldots, n_{r_i}^{(i)} \ge 0, 
    d_1^{(i)}, \ldots, d_{r_i}^{(i)} \ge 1\\
    n_{r_i+1}^{(i)}, \ldots, n_{r_i+s_i}^{(i)} \ge 1\\
    d_{r_i+1}^{(i)}, \ldots, d_{r_i+s_i}^{(i)} \ge 1}}
  \prod_{k=1}^{r_i+s_i} (d_k^{(i)})^{t_k^{(i)}}
\cdot q^{\sum_{k=1}^{r_i+s_i} n_k^{(i)} d_k^{(i)}}
  y_i^{\sum_{k=1}^{r_i} d_k^{(i)} 
    - \sum_{k=1}^{s_i} d_{r_i+k}^{(i)}}
$$
where $t_k^{i,j} \ge 1$ for all $k$ and $1 \le i < j \le N$, 
$t_k^{(i)} \ge 0$ for all $k$ and $1 \le i \le N$, 
\begin{eqnarray}    \label{yij1toNm>0.5}
& &\sum_{1 \le i < j \le N} \sum_{k=1}^{r^{i,j}+s^{i,j}} (t_k^{i,j} +1)
   + \sum_{1 \le i \le N} \sum_{k=1}^{r_i+s_i} (t_k^{(i)} +1)
   \nonumber   \\
&=&\sum_{1 \le i < j \le N, 1 \le \ell \le 2} (m_\ell^{i,j} + n_\ell^{i,j})
   + \sum_{1 \le i \le N} (m_i + n_i),
\end{eqnarray}
and $r, s \ge 0$ satisfy the inequalities
\begin{eqnarray}    \label{yij1toNm>0.6}
r+s \le \sum_{1 \le i \le N} \sum_{k=1}^{r_i+s_i} (t_k^{(i)} +1) 
\le \sum_{1 \le i \le N} (m_i + n_i).
\end{eqnarray}
It follows that the coefficient of ${\rm z^mv^n}$ in 
$\Coe_{y_1^0 \cdots y_N^0}(\mathfrak P_N^{a,b})$ 
is a finite $\Q$-linear combination of expressions of the form
$$ 
a^rb^s \cdot \prod_{1 \le i < j \le N}
\sum_{\substack{n_1^{i,j}, \ldots, n_{r^{i,j}}^{i,j} \ge 0, 
    d_1^{i,j}, \ldots, d_{r^{i,j}}^{i,j} \ge 1\\
    n_{r^{i,j}+1}^{i,j}, \ldots, n_{r^{i,j}+s^{i,j}}^{i,j} \ge 1,
    d_{r^{i,j}+1}^{i,j}, \ldots, d_{r^{i,j}+s^{i,j}}^{i,j} \ge 1}}
  \prod_{k=1}^{r^{i,j}+s^{i,j}} (d_k^{i,j})^{t_k^{i,j}} 
\cdot q^{\sum_{k=1}^{r^{i,j}+s^{i,j}} n_k^{i,j} d_k^{i,j}}
$$
$$
\cdot \prod_{1 \le i \le N}
\sum_{\substack{n_1^{(i)}, \ldots, n_{r_i}^{(i)} \ge 0, 
    d_1^{(i)}, \ldots, d_{r_i}^{(i)} \ge 1\\
    n_{r_i+1}^{(i)}, \ldots, n_{r_i+s_i}^{(i)} \ge 1,
    d_{r_i+1}^{(i)}, \ldots, d_{r_i+s_i}^{(i)} \ge 1}}
  \prod_{k=1}^{r_i+s_i} (d_k^{(i)})^{t_k^{(i)}}
\cdot q^{\sum_{k=1}^{r_i+s_i} n_k^{(i)} d_k^{(i)}}
$$
together with the condition that for every $1 \le i \le N$,
\begin{eqnarray*}
& &\sum_{1 \le j < i} \bigg (\sum_{k=1}^{r^{j,i}} d_k^{j,i} 
    - \sum_{k=1}^{s^{j,i}} d_{r^{j,i}+k}^{j,i} \bigg )
  + \bigg (\sum_{k=1}^{r_i} d_k^{(i)} 
    - \sum_{k=1}^{s_i} d_{r_i+k}^{(i)} \bigg )    \\
&=&\sum_{N \ge j > i} \bigg (\sum_{k=1}^{r^{i,j}} d_k^{i,j} 
    - \sum_{k=1}^{s^{i,j}} d_{r^{i,j}+k}^{i,j} \bigg ).
\end{eqnarray*}
By Proposition~\ref{Prop_t2t1yq1toNm>0} and \eqref{yij1toNm>0.5}, 
the coefficient of ${\rm z^mv^n}$ in 
$\Coe_{y_1^0 \cdots y_N^0}(\mathfrak P_N^{a,b})$ is contained in 
$\BD[a, b]$ with weight at most \eqref{yij1toNm>0.01}. 
Moreover, by \eqref{yij1toNm>0.6},
its degree in $a$ and $b$ is at most $\sum_{1 \le i \le N} (m_i + n_i)$.

(ii) Let $a=b=0$. By \eqref{yij1toNm>0.99}, \eqref{yij1toNm>0.100} and 
\eqref{yij1toNm>0.101}, the coefficient of 
$$
\prod_{1 \le i < j \le N, 1 \le \ell \le 2} (z_\ell^{i,j})^{m_\ell^{i,j}}
(v_\ell^{i,j})^{n_\ell^{i,j}}
$$ 
in 
$
\Coe_{y_1^0 \cdots y_N^0}(\mathfrak P_{N}^{0,0})
$ 
is a finite $\Q$-linear combination of expressions of the form
$$ 
\prod_{1 \le i < j \le N}
\sum_{\substack{n_1^{i,j}, \ldots, n_{r^{i,j}}^{i,j} \ge 0, 
    d_1^{i,j}, \ldots, d_{r^{i,j}}^{i,j} \ge 1\\
    n_{r^{i,j}+1}^{i,j}, \ldots, n_{r^{i,j}+s^{i,j}}^{i,j} \ge 1, 
    d_{r^{i,j}+1}^{i,j}, \ldots, d_{r^{i,j}+s^{i,j}}^{i,j} \ge 1}}
  \prod_{k=1}^{r^{i,j}+s^{i,j}} (d_k^{i,j})^{t_k^{i,j}} 
\cdot q^{\sum_{k=1}^{r^{i,j}+s^{i,j}} n_k^{i,j} d_k^{i,j}}
$$
together with the condition that for every $1 \le i \le N-1$,
$$
\sum_{1 \le j < i} \bigg (\sum_{k=1}^{r^{j,i}} d_k^{j,i} 
    - \sum_{k=1}^{s^{j,i}} d_{r^{j,i}+k}^{j,i} \bigg )
= \sum_{N \ge j > i} \bigg (\sum_{k=1}^{r^{i,j}} d_k^{i,j} 
    - \sum_{k=1}^{s^{i,j}} d_{r^{i,j}+k}^{i,j} \bigg ).
$$
Here $t_k^{i,j} \ge 1$ for all $k$ and $1 \le i < j \le N$, and 
$$
  \sum_{1 \le i < j \le N} \sum_{k=1}^{r^{i,j}+s^{i,j}} (t_k^{i,j} +1)
= \sum_{1 \le i < j \le N, 1 \le \ell \le 2} 
   (m_\ell^{i,j} + n_\ell^{i,j}).
$$
Now (ii) follows from a simplified version (with $a=b=0$) of 
the proof of (i).
\end{proof}

\bigskip\noindent
{\bf Contributions}

Zhenbo Qin is the sole author of this paper.

\bigskip\noindent
{\bf Conflict of Interest}

The author declares no competing interests.

\bigskip\noindent
{\bf Data Availability}

Data sharing is not applicable to this article as no datasets were generated or 
analyzed during the current study.

\end{document}